\documentclass[twoside,11pt]{article}
\usepackage{lipsum} 
\usepackage{amssymb,amsmath,amsthm,amsopn, leftidx} 
\usepackage{amscd} 
\usepackage{latexsym,graphicx}
\usepackage{xspace}
\usepackage{mathrsfs}
\usepackage[usenames,dvipsnames]{color}
\usepackage[implicit=true, 
  colorlinks=true,linkcolor=Black,citecolor=Black,urlcolor=Black]%
{hyperref}

{\begin{list}{}{%
   \vspace{-.8cm}
   \setlength{\rightmargin}{7mm}
   \setlength{\leftmargin}{7mm}}
    \item[]}
{\end{list}}

\makeatletter
\newcommand{\subjclass}[2][2010]{%
  \let\@oldtitle\@title%
  \gdef\@title{\@oldtitle\footnotetext{#1 \emph{Mathematics subject classification.} #2}}%
}
\newcommand{\keywords}[1]{%
  \let\@@oldtitle\@title%
  \gdef\@title{\@@oldtitle\footnotetext{\emph{Key words and phrases.} #1.}}%
}
\makeatother

\refstepcounter{equation}

\newtheorem{lem}{Lemma}
\newtheorem{theo}[lem]{Theorem}

\newtheorem{coro}[lem]{Corollary}
\newtheorem{prop}[lem]{Proposition}

\theoremstyle{definition}
\newtheorem{definition}[lem]{Definition}
\newtheorem{rem}[lem]{Remark}

\refstepcounter{lem}

\renewcommand{\descriptionlabel}[1]%
     {\hspace{\labelsep}\textsf{#1}}

\newtheorem{theorem}[lem]{Theorem}

\newtheorem*{lemma*}{Lemma}
\newtheorem*{theorem*}{Theorem}

\theoremstyle{remark}

\newcommand{\mycaption}[1]{\begin{quote} 
 \caption{\small#1\normalsize}\end{quote}\vspace{-3em}}

\newcommand\f{{\bf f}}
\newcommand\bF{{\mathbb F}} 

\newcommand\C{{\mathbb C}}

\newcommand\U{{\mathbb U}}
\newcommand\T{{\mathbb T}}
\newcommand\Z{{\mathbb Z}}
\newcommand\R{{\mathbb R}}
\newcommand\I{{\mathcal I}}
\newcommand\bP{{\mathbb P}}
\newcommand\bbar{\overline}

\newcommand\biot{{\boldsymbol\iota}}

\newcommand\A{{\mathcal A}}

\newcommand\cC{{\mathscr C}}

\newcommand\Aut{{\rm Aut}}

\newcommand\ssm{{\smallsetminus}}

 at 11truept

\def\bet{{\boldsymbol\eta}}
\def\brh{{\boldsymbol\rho}}
\def\bmu{{\boldsymbol\mu}}

\def\cC{{\mathcal C}}
\def\bT{{\mathbb T}}
\def\bP{{\mathbb P}}
\def\Q{{\mathbb Q}}
\def\cH{{\mathcal H}}
\def\p{{\bf p}}
\def\q{{\bf q}}
\def\r{{\bf r}}
\def\s{{\bf s}}
\def\o{{\bf o}}
\def\msk{\medskip}
\def\bsk{\bigskip}
\def\ssk{\smallskip}
\def \A {{\bf A}}
\def \M {{\bf M}}
\def \X {{\bf X}}
\def \U {{\bf U}}

\input xy
\xyoption{all}

\begin{document}

\setcounter{footnote}{1}
\setcounter{equation}{0}

\title{On Real and Complex Cubic Curves$^1$}

\author{Araceli Bonifant$^2$ and John Milnor$^3$}

\footnotetext[1]{To appear in L'Enseign. Math.}
\footnotetext[2]{Mathematics Department, University of Rhode Island; e-mail:\hfill{\mbox{}}\\ 
\hbox{bonifant@uri.edu}}
\footnotetext[3]{Institute for Mathematical Sciences, Stony Brook University;
e-mail: \hfill{\mbox{}}\\\hbox{jack@math.stonybrook.edu }}

\date{}


\maketitle

\begin{abstract}
An expository description of smooth cubic curves in the real or
 complex projective  plane. 
\end{abstract}

\setcounter{lem}{0}
\setcounter{footnote}{2}

\section{Introduction.} This note will present 
elementary proofs for basic facts about smooth cubic curves $\cC$
in the complex projective plane $\bP^2(\C)$, or the corresponding curves
$\cC_\R$ in the real projective plane $\bP^2(\R)$ when the defining
 equation has real coefficients.
The presentation will center around two
different normal forms, which we refer to as the
 \textbf{\textit{Hesse normal form}}
\begin{equation}\label{E-H} x^3+y^3+z^3~=~ 3\,k\, x\,y\,z~ \end{equation}
(using homogeneous coordinates), and the \textbf{\textit{standard normal form}}
\begin{equation}\label{E-W} y^2~=~ x^3+ax+b \end{equation}
(using affine coordinates $(x:y:1)$ with $z=1$).
We are particularly concerned with classification up to 
\textbf{\textit{projective equivalence}}
(that is up to a linear change of coordinates in the projective plane).
The set of \textbf{\textit{flex points}} (three in the real case and nine
 in the complex case) plays a central role in our exposition.

Although much of the material which follows is well known, there are a
few things which we have not been able to find in the literature.
One of these is the following concise classification
(see Theorem \ref{T-realH}):

\begin{quote}\it Every smooth irreducible real cubic curve ${\mathcal C}_\R$
 is real projectively equivalent to one and only one curve $\cC(k)_\R$ in the 
Hesse normal form. Here any real parameter $k\ne 1$ can occur. The
curve $\cC(k)_\R$ is connected if $k<1$, and  has two 
components if $k>1$.
\end{quote}

 Another is the precise description of the automorphism group, 
consisting of projective transformations which map the curve to itself. 
This has order 6 in the real case, 
and order 18 for a generic complex curve; but has order
36 or 54 in the special case of a complex curve which has square or hexagonal
symmetry. In  all cases it has a maximal
abelian subgroup which acts freely on the curve, and acts transitively on 
its set of flex points. (See Corollaries \ref{C-aut2} and \ref{C-realaut}.)
On the other hand, the group of birational automorphisms,   
which is canonically isomorphic to the group of conformal automorphisms, 
acts transitively on the entire curve (Corollaries \ref{C-aut3} 
and~\ref{C-birat}). 

One useful elementary remark is that the projective
equivalence class of a curve in the standard normal form is uniquely
determined by the \textbf{\textit{shape}}
 of the ``triangle'' in the complex $x$-plane
formed by the three roots of the equation $x^3+ax+b=0$. (See Figure 
\ref{F-jpic}, as well as Definition \ref{D-shape} and Proposition \ref{P-tri}.)
\medskip

The paper is organized as follows.
Sections \ref{S-H} through \ref{S-CT} concentrate on the complex case
(although some arguments work just
as well over an arbitrary subfield $\bF\subset\C$).
 Section~\ref{S-H} studies flex points, reduction to Hesse normal form,
 and provides a preliminary description of the automorphism group.
Section~\ref{S-class} studies reduction to the standard  normal form,
as well as the $J$-invariant
$$ J(\cC)~=\frac{4a^{\;3}}{4a^{\;3} + 27\,b^{\;2}}~,$$
and the computation of $J$ as a function of the Hesse parameter $k$.
Section~\ref{S-conf} discusses the conformal classification of $\cC$ as
a Riemann surface, and shows that a conformal 
diffeomorphism from $\cC$ to $\cC'$ extends to a projective automorphism
of $\bP^2$ if and only if it maps flex points to flex points.
Section~\ref{S-CT} describes the chord-tangent map
 $\cC\times\cC\to\cC$ and the related additive group structure on the curve.
Finally, Section~\ref{S-R} describes real cubic curves. In particular,
it provides a canonical affine picture, so that the automorphisms are
clearly visible, and so that any two real curves can be directly compared.
(See Figure \ref{F-canpics}.)\bigskip

{\bf Notation.} We use the notation $(x,y,z)$ for a non-zero
point of the  complex 3-space $\C^3$, and the notation $(x:y:z)$ for the
corresponding point of $\bP^2\!$, representing the equivalence
class consisting of all multiples $(\lambda\,x,\,\lambda\,y,\,\lambda\,z)$
with $\lambda\in\C\ssm\{0\}$. However, it is sometimes convenient to represent
a point of $\bP^2$ by a single bold letter such as $\p$. Note that any linear
automorphism of ${\mathbb C}^3$ gives rise to a
 \textbf{\textit{projective automorphism}} of the projective plane
 ${\mathbb P}^2$.
\bigskip

{\bf Historical Remarks.}
Hesse's actually used a constant multiple of our $k$ as parameter. 
Our ``standard normal form'' is a special case of a form used
 by Newton, and is a close relative of the form which Weierstrass 
introduced much later. Our $J(\cC)$ is
just Felix Klein's invariant $j(\cC)$ divided by 1728. (For the original
papers, see  \cite{Ne}, \cite[H2]{H1},  
\cite[W2]{W1}, and \cite{Kl}.) The ``tangent'' part of the chord-tangent
map was used by \hbox{Diophantus}
 of Alexandria in the third century to construct
new rational points on a cubic curve from known ones,  although
this was done in a purely algebraic way.   More 
 than thirteen centuries later, in the 1670's Newton used the 
``chord-tangent construction'' to
 interpret the solutions of Diophantine equations given by Diophantus and 
Fermat. (Compare \cite[Section 11.6]{Sti}.) Weil says that the chord process
was first used by Newton, although Bashmakova claims that it was used already
by Diophantus. (See \cite{Weil} and \cite{Ba}.)
The closely related additive structure is  due to Poincar\'e \cite{P}, who 
was the first to study the arithmetic of algebraic curves. (Compare 
\cite{Kn} and \cite[p.~412]{Ba}.) Real
cubic curves in the affine plane were studied by Newton. 
(Compare \cite{Ne}, as well as \cite{BK}.) For
further historical remarks, see \cite{AD}, \cite{Dol}, and \cite{RB}. For
an elementary introduction to the field see \cite{Gib};
and for real cubic curves from an older point of view, see \cite{Wh}.
\bigskip

\section{Hesse Normal Form for Complex Cubic Curves. }\label{S-H}
This section will be concerned with the work of Otto Hesse 
and its consequences. (See \cite[H2]{H1}, both published in 1844.) 
Hesse introduced\footnote{Gibson refers to the $\cC(k)$ as
 ``Steiner curves'', presumably referring to \cite{St}.}
the family of cubic curves $\cC(k)$ consisting of all points 
 $(x:y:z)$ in the projective plane $\bP^2=\bP^2(\C)$ which satisfy the
homogeneous equation $\Phi_k(x,y,z)=0$ where 
\begin{equation}\label{E-hnf}
\Phi_k(x,y,z)~=~x^3+y^3+z^3~-~3\,k\,x\,y\,z~.
\end{equation}
Given a generic point $(x:y:z)$ in $\bP^2$, we can solve the equation 
$\Phi_k(x,y,z)=0$ for
$$k~=~\frac{x^3+y^3+z^3}{3\,x\,y\,z}~\in~\C\cup\{\infty\}.$$
(Thus for $k=\infty$ we define $\cC(\infty)$ to be the 
 locus $xyz=0$.) However, there are nine
 \textbf{\textit{exceptional points}}\footnote{We will see in 
Remark \ref{R-hess-flex} that these nine ``exceptional points'' 
on any smooth $\cC(k)$ are precisely the nine flex points.}
where
$${\rm both}\qquad x^3+y^3+z^3~=~0\qquad{\rm and}\qquad x\,y\,z~=0~,$$
so that
the parameter $k$ is not uniquely defined. All of the curves $\cC(k)$
pass through all of these nine points, which have the form
\begin{equation}\label{E-hess-flex}
 (0:1:-\gamma)\quad{\rm or}\qquad(-\gamma:0:1)\quad{\rm or}
\quad(1:-\gamma:0)\quad{\rm with}\quad \gamma^3=1~.
\end{equation}
In the complement of these nine exceptional points, the space $\bP^2(\C)$
is the disjoint union of the Hesse curves.
\bigskip

\begin{definition}\label{D-sing} 
A point of the curve  $~\Phi(x,y,z)~=~0~$ is  \textbf{\textit{singular}}
if the partial derivatives $\Phi_x,\,\Phi_y$ and $\Phi_z$ all vanish at the
point. The curve is called \textbf{\textit{smooth}} if it has no singular
points.\end{definition} \medskip

\begin{lem}\label{L-Hsing} The Hesse curve $\cC(k)$ has
 singular points if and only if  either $k^3=1$ or $k=\infty$.
\end{lem}
\smallskip

\begin{proof}
 For the curve $\cC(k)$ with $k$ finite, a singular point must satisfy
 the equations
$$ x^2=kyz~,\quad y^2=kxz\,,\quad z^2=kxy~,$$
which imply that $x^3=y^3=z^3=k\,xyz$, and hence $x^3y^3z^3= k^3x^3y^3z^3$.
For $k^3\ne 1$ with $k$ finite, it follows easily that there are no
 singularities. On the other hand, if $k^3=1$, then it is not hard to check 
that $\cC(k)$ is the union of three straight lines of the form
$\alpha x+\beta y+z=0$ with $\alpha^3=\beta^3=1$ and $\alpha\beta=k$.
Hence it is singular at the three points where two of these
lines intersect.\footnote{For $k=1$, one such intersection point $(1:1:1)$
is clearly visible to the upper right in Figure \ref{F-hess-fol}, even though 
the rest of the
two intersecting lines lie outside of the real projective plane.}
Similarly the curve $\cC(\infty)$ is clearly singular at the three points
 $(0:0:1)$, $(1:0:0)$ and $(0:1:0)$ where two of the lines $x=0$, $y=0$, 
and $z=0$ intersect. 
\end{proof}
\medskip

Thus altogether there are twelve points in $\bP^2$ which are singular
for one of these curves. If we remove
these twelve singular points and also the nine exceptional points from
$\bP^2$, then we obtain a smooth foliation by Hesse curves.

\bigskip

\begin{figure}[!ht] 
\begin{center}
\includegraphics[width=3.5in]{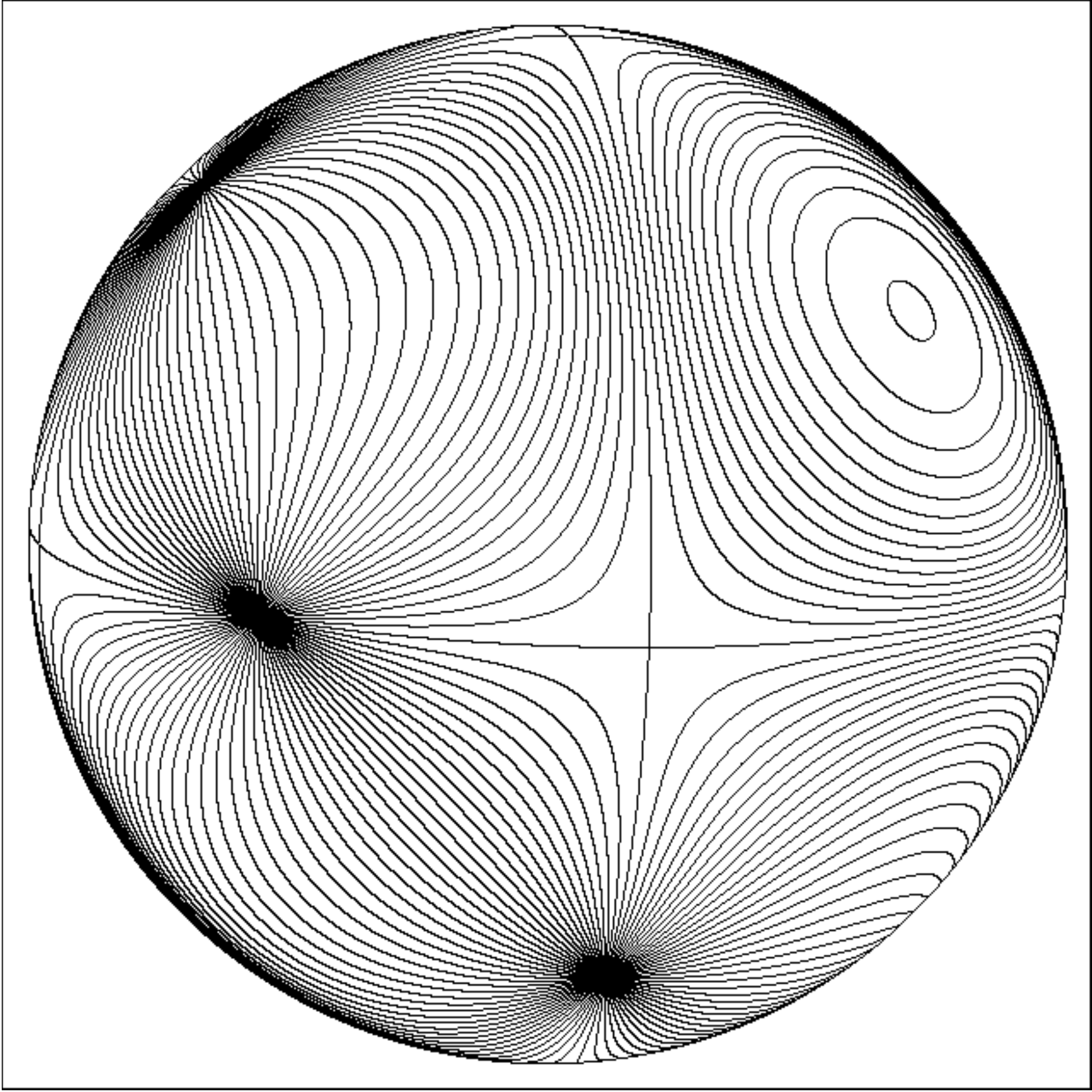}
\mycaption{\label{F-hess-fol} \sf ``Foliation'' of the real projective plane
 by the Hesse pencil 
of curves $\cC(k)_\R=\cC(k)\cap\bP^2(\R)$. 
Here $\bP^2(\R)$ is represented
as a unit sphere with antipodal points identified. Note the three exceptional
points $(-1:1:0)$, $(0:-1:1)$, and $(1:-1:0)$ where all of the 
$\cC(k)_\R$ intersect.
Also note the three singular points where the components of\break
\hbox{$\qquad\qquad\cC(\infty)_\R\,=\,\{x=0\}\cup\{y=0\}\cup\{z=0\}\qquad\qquad$} \break
intersect, and note the isolated singular point at $(1:1:1)\in \cC(1)_\R$. 
The figure has $120^\circ$ rotational symmetry about this
point. (This figure has been borrowed from our paper \cite{BDM}, which studies
rational maps preserving such cubic curves.)}
\end{center}
\end{figure}
\smallskip

\begin{definition}\label{D-aut} Let $\Aut(\bP^2)$ be the group of all 
projective 
automorphisms of $\bP^2$; and for any curve $\cC\subset\bP^2$ let
$\Aut(\bP^2,\,\cC)$ be the subgroup consisting of projective
automorphisms which map $\cC$ onto itself. 
\end{definition} \smallskip

Curves defined by the Hesse equations  (\ref{E-hnf}) are clearly highly
symmetric. Following is a precise statement.
\smallskip

\begin{lem}\label{L-hess-aut} The group $\Aut\big(\bP^2({\mathbb C})\big)$ 
contains an
 abelian subgroup $$N\cong\Z/3\oplus\Z/3$$  independent of $k$, which acts 
without fixed points on every smooth $\cC(k)$,
and acts simply transitively on the set of nine ``exceptional points'' of
equation $(\ref{E-hess-flex})$. The group $\Aut(\bP^2)$
also contains an element $\biot$ of order two, which maps each $\cC(k)$
onto itself, and such that conjugation by $\biot$
maps each element of $N$ to its inverse.
\end{lem}
\smallskip

Thus the automorphism group ${\Aut}(\bP^2,\,\cC(k))$ contains
 at least 18 elements.
In fact, we will show in Theorem \ref{T-hnf} below that any smooth cubic curve 
$\cC$ can
 be put into the form (\ref{E-hnf}), so that the automorphism
group $\Aut(\bP^2,\,\cC)$ always contains a corresponding  18 element
subgroup. (In most cases, this is the full automorphism group,
but for some special curves there are extra symmetries, as described in
Corollary \ref{C-aut2}.)
\medskip

\begin{proof}[Proof of Lemma \ref{L-hess-aut}] A cyclic permutation of the
three coordinates $x,\,y,\,z$
 clearly acts effectively on every curve of the form (\ref{E-hnf}). 
This action has just one fixed point $(1:1:1)\in\cC(1)\subset\bP^2$, but has
no fixed points in $\cC(k)$ for $k\ne 1$.
If $\gamma$ is a primitive cube root of unity, then  the transformation
$$ (x:y:z)~~~\mapsto~~~(x:\gamma \,y:\gamma^2\,z)$$
is another automorphism of order three which
 commutes with the cyclic  permutation
of coordinates. It is not difficult to check\footnote{Here is a typical
case. If $(x:y:z)=(y:\gamma z: \gamma^2 x)$ then, using the fact that\break
$(x:y)=(u:v)$ if and only if $xv=yu$, we can check that 
$x^2=\gamma yz,~y^2=\gamma xz$ and $z^2=\gamma xy$. Therefore
$x^3=y^3=z^3=\gamma xyz$, hence $k=\gamma$ and the curve is singular.}
 that the abelian group $N$
generated by these two transformations has fixed points only on the
singular Hesse curves with $k^3=1$ or $k=\infty$.
Furthermore, it acts simply transitively on the set of exceptional points 
(\ref{E-hess-flex}). Finally, the permutation 
$(x:y:z)\leftrightarrow(y:x:z)$ is an element $\biot$ of order two
 which carries each $\cC(k)$ to itself, and has the required action
on $N$. (This permutation
does have four fixed points on each $\cC(k)$, consisting of just one
 of the nine 
exceptional points, namely $(1:-1:0)$, together with the three
points of  $\cC(k)$ which lie on the line $x=y$.) 
\end{proof}
\bigskip

Although we are primarily interested in cubic curves, we will often
use results which apply to curves of any degree $n\ge 2$.
The most fundamental property of complex projective curves of specified degree
is the following:

\begin{quote}\it
For any smooth curve $\cC\subset\bP^2(\C)$ of degree $n\ge 2$ and any
line $L\subset\bP^2$, the intersection $\cC\cap L$ consists of $n$ points,
counted with multiplicity.\footnote{This is an special case of B\'ezout's 
theorem; but can also be proved just by restricting
 the defining equation $\Phi(x,y,z)=0$ to
the line $L$, and then using the Fundamental Theorem of Algebra.}
\end{quote}

\noindent Here:
\begin{itemize}
\item A transverse intersection has multiplicity one.
\item The intersection between $\cC$ and its tangent line at a generic
point has multiplicity two.
\item At certain special ``flex  points'' this tangential intersection will
have multiplicity three or more. 
\end{itemize}

\noindent For a cubic curve, note that the intersection multiplicity is three
if and only if there are no other points of intersection between
$\cC$ and $L$.\medskip

\begin{definition} For a curve $\cC$ of any degree, a non-singular point
is called an inflection point, or briefly a  \textbf{\textit{flex point}},
if the intersection multiplicity between $\cC$ and its
tangent line is three or more. (Of course in the cubic case, 
this intersection multiplicity is always precisely three, 
 unless $\cC$ contains an entire straight line, which implies that $\cC$
contains singular points.)
\end{definition}
\medskip

We will prove in Theorem \ref{T-9flex} that every smooth complex 
cubic curve has precisely nine
flex points. However, the proof will be based 
 on a more general discussion which applies
to smooth curves of any degree.

\begin{definition} Let $\Phi(x,y,z)$ be a homogeneous
polynomial of degree $n$.
The associated \textbf{\textit{Hessian determinant}}
is the  homogeneous polynomial function
\begin{equation}\label{E-hessian}
{\mathcal H}_\Phi(x,y,z)~=~  {\rm det}\!\left(\begin{matrix}
\Phi_{xx}& \Phi_{xy}& \Phi_{xz}\\ \Phi_{yx}& \Phi_{yy}& \Phi_{yz}\\
\Phi_{zx}& \Phi_{zy}& \Phi_{zz}\end{matrix}
\right),
\end{equation}
of degree $3(n-2)$, where the subscripts on the right indicate partial\break
 derivatives.
\end{definition}
\msk

\begin{theorem}\label{T-hess-d} \it Let $\cC$ be any  smooth curve
 of degree three or more with defining equation $\Phi(x,y,z)=0$. Then
the set ${\rm Flex}(\cC)$
consisting of all flex points in $\cC$ is equal to the set of all points
in $\cC$ which satisfy the homogeneous equation $\quad \cH_\Phi(x,y,z)=0~.$
\end{theorem}\msk

As the first step in the proof, we must show that this locus $\cH_\Phi=0$
behaves properly under projective transformations. Let 
\begin{equation}\label{E-A}
 (x,y,z)~\mapsto A(x,y,z) = (u,v,w)
\end{equation}
be a non-singular linear transformation, and let $A_*:\bP^2(\C)\to\bP^2(\C)$
be the induced projective transformation. Note that  $A_*$ maps
the curve $\cC$ defined by the equation $\Phi(x,y,z)=0$ to the curve 
$\cC'=A_*(\cC)$ defined by the equation $\Psi(u,v,w)=0$, where
 $\Psi=\Phi\circ A^{-1}$ (or equivalently $\Phi=\Psi\circ A$),
as one sees from the diagram
$$
\xymatrix@R=3.8pc@C=3.8pc@M=0.2pc{{\mathbb C}^3 \ar[r]^A \ar[rd]^\Phi &{\mathbb C}^3\ar[d]^\Psi\\ &{\mathbb C}~.} $$

\begin{lem}\label{L-hess}
With these notations, $A_*$ maps the curve defined by the equation
$\cH_\Phi(x,y,z)=0$ to the curve defined by $\cH_\Psi(u,v,w)=0$. In particular,
it maps the locus of points on $\cC$ with $\cH_\Phi(x,y,z)=0$ to the locus
 of points on $A_*(\cC)$ satisfying $\cH_\Psi(u,v,w)=0$.
\end{lem}
\smallskip

\begin{proof}
For this proof only, it will be convenient to switch to matrix\break notation,
representing a point in $\C^3$ by a column vector $\bf X$, and writing a
linear change of coordinates $A:\C^3\to\C^3$ as 
$$\X~~\mapsto~~ \A \X =\U\qquad{\rm where}\qquad
 \X =\left[\begin{matrix}x\\y\\z\end{matrix}\right]\qquad{\rm and}\qquad
\U=\left[\begin{matrix}u\\v\\w\end{matrix}\right]~,$$
and where $\A$ can be any non-singular $3\times 3$ matrix.
Let $\M_\Phi(\X )$ be the matrix of second partial derivatives of $\Phi$,
with determinant $\cH_\Phi(\X )$, and define $\M_\Psi(\U)$ with determinant
$\cH_\Psi(\U)$ similarly. Then we will prove that
\begin{equation}\label{E-hess2}
\M_\Phi(\X ) ~=~\A^t\,\M_\Psi(\A \X )\,\A~,
\end{equation}
where $\A^t$ is the transpose matrix.

To prove (\ref{E-hess2}), note that the Taylor series for $\Phi$ at a point
$\X _0$ has the form $$\Phi(X_0+X)~=~\Phi(X_0)~+~({\rm linear~terms})~+~
({\rm quadratic~terms})~+~({\rm higher~order~terms})~,$$
where the quadratic terms can be written as ${\textstyle\frac{1}{2}}\X ^t
\, \M_\Phi(\X _0)\,\X $.
 If we ignore all terms of degree other than two, then this can be 
written briefly as
\begin{equation}\label{E-mat}
 \Phi(\X _0+\X )~=~\cdots~+~{\textstyle\frac{1}{2}}\X ^t\, \M_\Phi(\X _0)\, \X ~+~\cdots~.
\end{equation}
Similarly
$$\Psi(\U_0+\U)~=~\cdots~+~{\textstyle\frac{1}{2}}\U^t\, \M_\Psi(\U_0)\, \U~+~\cdots~.$$
Now substituting $\A \X $ for $\U$ and $\A \X _0$ for $\U_0$, and
recalling that\break $\Phi=\Psi\circ \A$, this last equation takes the form
$$\Phi(\X _0+\X )~=~\Psi(\U_0+\U)~=~
\cdots~+~{\textstyle\frac{1}{2}}\X ^t\A^t\, \M_\Psi(\A \X _0)\,\A \X ~+~\cdots~.$$
Comparing this expression with (\ref{E-mat}), and noting that the two equations
must agree for all $\X $, the required equation (\ref{E-hess2}) follows.

Now taking the determinant of both sides of (\ref{E-hess2}) 
and switching back to non-matrix notation, we obtain the identity
\begin{equation}\label{E-hess-ch}
 \cH_\Phi(x,y,z)~=~{\rm det}(A)^2\cH_\psi(u,v,w)~,\quad{\rm where}\quad
(u,v,w)=A(x,y,x)~. \end{equation}
Lemma \ref{L-hess} now follows easily, since the constant factor
${\rm det}(A)^2$ does not affect the induced transformation of projective
 space. \end{proof}
\medskip

\begin{proof}[Proof of Theorem \ref{T-hess-d}]
We must show that a point $(x_0:y_0:z_0)\in\cC$ is a flex point if and only if
$\cH_\Phi(x_0,y_0,z_0)=0$. Choose a linear change of coordinates which maps
the given point $(x_0,y_0,z_0)$ to $(0,0,1)$. After a rotation of the
 $x,y$ coordinates, we may assume that the image curve is tangent to the
line $y=0$. Now, working with affine coordinates\break $(x:y:1)$, 
we can solve locally for $y$ as a smooth function \hbox{$y=f(x)$}, where the
derivative $dy/dx=f'(x)$ vanishes for $x=0$. Differentiating the equation
$\Phi(x,\,f(x),\,1)=0$ twice, we obtain
$$ \Phi_x+\Phi_y f'(x)=0\qquad {\rm and}\qquad
\Phi_{xx} + 2\Phi_{xy}f'(x)+\Phi_{yy}(f'(x))^2+\Phi_yf''(x)=0 $$
along the curve, where $\Phi_x(0,0,1)=0$ but $\Phi_y(0,0,1)\ne 0$.
In particular, at the specified point with $x=f'(x)=0$, we see that
$$ \Phi_{xx}=0\quad\Longleftrightarrow\quad f''(x)=0\quad
\Longleftrightarrow\quad x~~{\rm is~a~flex~point}~.$$

To finish the proof of Theorem \ref{T-hess-d}, we must show,
 with a choice of coordinates as above,
that $\Phi_{xx}(0,0,1)=0$ if and only if $\cH_\Phi(0,0,1)=0$.
Note that $\Phi(x,y,z)$ can be written uniquely as a sum of monomials
$c_{ijk}x^iy^jz^k$ with $i+j+k=n$. Since $\Phi(0,0,1)=0$, we know that the
coefficient of $z^n$ is zero, and it follows easily that $\Phi_{zz}(0,0,1)=0$.
Similarly, since $\Phi_x(0,0,1)=0$, it follows easily that $\Phi_{xz}(0,0,1)=0$;
but $\Phi_y(0,0,1)\ne 0$ hence $\Phi_{yz}(0,0,1)\ne 0$. It is now
 straightforward to check that
Hessian determinant reduces to $\cH_\Phi=-\Phi_{xx}\Phi_{yz}^{\,2}$ 
at the point $(0,0,1)$, and the conclusion follows.
\end{proof}
\medskip

\begin{rem}\label{R-y''}
More explicitly,  whenever
 $\Phi_y(0,0,1)\ne 0$ so that $y$ can be expressed locally as a
 smooth function $y=f(x)$, we have the identity
\begin{equation*}
 \cH_\Phi(0,0,1)~=~4\,\Phi_y^{\,3}(0,0,1)\,f''(0)~.
\end{equation*}
In the case where $f'(0)=0$, this follows from
the proof above, together with the observation that
 $\Phi_{yz}(0,0,1)=2\,\Phi_y(0,0,1)$. (This
last equality can be checked by comparing the first derivative of
the monomial $yz^2$ with  respect to  $y$, and the second derivative
 with respect to $y$ and $z$.) 

The more
general case where $f'(x)\ne 0$ can be dealt with by noting that the change of
coordinates $$(x,y,z)\mapsto(x,\, ax+y,\,z)$$ clearly does not affect
$d^2y/dx^2$, 
and by noting that the locus $\cH_\Phi=0$ transforms by 
equation~(\ref{E-hess-ch}).

Even more generally we can write
\begin{equation}\label{E-H=}
\cH_\Phi(x,f(x),1) ~=~4\,\Phi_y^{\,3}(x,f(x),1)\,f''(x)
\end{equation}
at any point of the curve where $\Phi_y\ne 0$,
since both sides of this equation are invariant under translation.

\end{rem}
\medskip

\begin{theorem}\label{T-9flex} \it
Every smooth cubic curve in $\bP^2(\C)$ has nine flex points.
\end{theorem}

\begin{proof} If a curve $\cC_1$ of degree $n_1$ and a curve
$\cC_2$ of degree $n_2$  have only smooth transverse intersections, then
it follows from B\'ezout's Theorem that the 
number of intersection points is precisely equal to the product $n_1n_2$.
(See for example \cite[Section~6.1]{BK} or  \cite[pp. 36, 54]{Ha}.) 
We are intersecting two curves
$\{\Phi=0\}$ and $\{\cH_\Phi=0\}$ which both have degree three.
Thus to prove Theorem  \ref{T-9flex},  we need only show
 that these two curves have only smooth transverse intersections.\footnote{
The curve $\{\cH_\Phi=0\}$ may have singular points, but is always
non-singular near the intersection.}

Since  equation (\ref{E-H=}) holds at all points of $\cC$ with 
$\Phi_y\ne 0$, we can differentiate with respect to $x$ to obtain
$$\frac{\partial\cH_\Phi}{\partial x}(0, 0,1)
~=~\Phi_y^{\;3}(0, 0,1)\,f'''(x) $$
whenever $f'(0)=0$. But at a flex point of a smooth cubic curve,
where $f''(0)=0$, the third derivative  $f'''(0)$ can never vanish,
  for this would imply that the entire tangent line at that point would
have to be contained in $\cC$. Thus $\partial\cH_\Phi/\partial x\ne 0$
at a flex point; and
it follows easily that the two curves $\{\cH_\phi=0\}$ and $\{\Phi=0\}$ have a
 transverse intersection at every flex point. Thus every smooth cubic
 curve has exactly nine flex points; which completes the proof
of Theorem \ref{T-9flex}.\end{proof}
\bsk

\begin{rem}\label{R-hess-flex}
In the special case of a smooth Hesse curve
 $\cC(k)$, the nine flex points coincide
with the nine ``exceptional points'' of equation (\ref{E-hess-flex}).
To see this,
taking $~\Phi(x,y,z)=x^3+y^3+z^3-3\,k\,x\,y\,z$, note for example that
 $\Phi_{xx}=6\,x$
and $\Phi_{xy}=-3\,k\,z$. A straightforward computation shows that the Hessian
determinant is given by 
$$\cH_\Phi(x,y,z)~=~3^3\big((8-2\,k^3)x\,y\,z~-~2k^2(x^3+y^3+z^3)\big)~.$$
If $\Phi(x,y,z)=0$, then we can substitute $3kxyz$ for $x^3+y^3+z^3$ on the
 right side of this equation. This yields $\cH_\Phi(x,y,z)=6^3(1-k^3)xyz$.
If we are in the non-singular case, with $k^3\ne 1$, then it follows that
$\Phi$ and $\cH_\Phi=0$ are both zero only at the nine points with
$$x^3+y^3+z^3~=~x\,y\,z~=~0~,$$
as in (\ref{E-hess-flex}). (On the other hand, if $k^3=1$ then the
Hessian is identically zero on $\cC(k)$, which means that $\cC(k)$ is
a union of straight lines.)

The set of nine flex points together
with the twelve lines joining them form a fascinating configuration. 
(Compare Figure \ref{F-flexgrid} and Remark \ref{R-Hconfig}.)
\end{rem}
\medskip

\subsection*{\bf Reduction to Hesse Normal Form.}
 Recall that, two algebraic varieties 
in a projective space ${\mathbb P}^n$  are 
\textbf{\textit{projectively equivalent}} 
if there exists a projective automorphism of ${\mathbb P}^n$
 which carries one variety onto the other.
\msk

 The following result is taken from a textbook
 published by Heinrich Weber in 1898. (See \cite[v.3, p.22]{Web}. 
We don't know whether this result was known earlier.)
\ssk 

\begin{theo}\label{T-hnf}
\it Every smooth cubic curve in
$\bP^2(\C)$ is projectively equivalent to a curve $\cC(k)$ in the Hesse normal
form
$$x^3+y^3+z^3~=~3\,k\,x\,y\,z~,$$
with $k\in {\mathbb C}$, $k^3\ne 1$.
\end{theo}\medskip

{\bf Proof.} Choose two distinct flex points for the given curve $\cC$,
 and choose homogeneous coordinates $x,\,y,\,z$ so that:

\begin{itemize}
\item[$\bullet$] the line $x=0$ is tangent to $\cC$ at the first flex point,

\item[$\bullet$] the line $y=0$ is tangent at the second flex point, and

\item[$\bullet$] the line $z=0$ passes through both flex points. 
\end{itemize}

\noindent Working in affine coordinates with $z=1$, these conditions mean
that $\cC$ has no finite point on the lines $x=0$ or $y=0$. In other words
the polynomial function $\Phi(x,y,1)$ must take a non-zero constant value
on these two lines; say $\Phi(x,y,1)=1$ whenever $xy=0$. Hence it must have the
form $\Phi(x,y,1)= 1+xy(ax+by+c)$. In homogeneous coordinates, this means that
$$ \Phi(x,y,z)~=~ z^3 \,+\,xy(ax+by) +cxyz~.$$
Furthermore $a\ne 0$ since otherwise $(1:0:0)$ would be a singular point,
and $b\ne 0$ since otherwise $(0:1:0)$ would be singular.
Now to put this polynomial in the required form, we must express $xy(ax+by)$
as a sum of two cubes.\footnote{More generally, any smooth cubic locus
$\Psi(x,y)=0$ in $\bP^1(\C)$ is just a union of three
distinct points, and it is not hard to choose a projective equivalence
(= fractional linear transformation) from the Riemann sphere $\bP^1(\C)$
to itself which carries one such triple to any other.}
 In fact, consider the identity
$$ (\gamma\,p\,x+q\,y)^3\,+\,(-p\,x-\gamma\,q\,y)^3~=~  3i\sqrt{3}\,
xy\big(-p^2q\,x\,+\,pq^2\,y\big)\,, $$
where $\gamma=(-1+i\sqrt{3})/2$. It is not difficult
to choose $p$ and $q$ so as to satisfy the required equalities
$$  -3i\sqrt{3}\,p^2q=a\qquad{\rm and}
\qquad 3 i\sqrt{3}\,pq^2=b~,$$
since we can first solve for $p/q=-a/b$, and then solve for $p$. \qed
\bigskip

\begin{rem}\label{R-not-unique}
According to Lemma~\ref{L-Hsing}, a curve in Hesse form, with $k$ finite,
 is smooth if and only if $k^3\ne 1$. This Hesse form is not unique, since
there are several different ways of choosing the two flex points. We will see
in Theorem \ref{k-sym} that, for a generic choice of the smooth curve $\cC$,
 there are twelve different possible choices
of the parameter $k$.
\end{rem}\ssk

\begin{coro}\label{C-aut} Every smooth complex cubic curve possesses an 
automorphism group of order at least 18 which acts transitively on its set of
 nine  flex points.\end{coro}

\begin{proof} This follows immediately by
 combining Lemma~\ref{L-hess-aut} and Remark~\ref{R-hess-flex} with 
Theorem~\ref{T-hnf}.
\end{proof}

(For a more precise description of the automorphism group, see\break Corollary
\ref{C-aut2}.)

\setcounter{lem}{0}
\section{The Standard Normal Form.}\label{S-class}  
Recall that a curve in standard normal form is defined by the equation
$$ y^2~=~x^3+ax+b$$
in affine coordinates $(x:y:1)$. Equivalently, using homogeneous coordinates
$(x:y:z)$, it is defined by the equation $\Phi=0$ where
\begin{equation}\label{E-snf}
\Phi(x,y,z)~=~ -y^2z+x^3+axz^2+bz^3~.
\end{equation}
One virtue of this normal form is that is useful over many different
 fields.\footnote{More precisely, one can reduce to this normal form over any
field of characteristic other than two or three.}
For our purposes, the following level of generality will be convenient.
\ssk

Let $\bF\subset\C$ be any subfield of the complex numbers. A curve $\cC
\subset\bP^2(\C)$ is said to be \textbf{\textit{defined over $\bF$}}
if it is defined by a homogeneous polynomial
equation $\Phi(x,y,z)=0$  with coefficients in $\bF$. Similarly, an
$\bF$-{\textbf{\textit{projective transformation}}} will mean a projective
 transformation with coefficients in $\bF$, or equivalently one which
 maps the projective space $\bP^2(\bF)$ onto itself.

The notation $\cC_\bF\subset\bP^2(\bF)$ will be used for the intersection
$\cC\cap\bP^2(\bF)$, 
consisting  of all points of $\cC$ with coordinates in $\bF$.
\smallskip

{\bf Caution.} In this generality,
 there is no guarantee that $\cC_\bF$ will have any
points at all. For example, the equation 
 $3x^3+4y^3+5z^3=0$ has no non-zero solution with $x,y,z$ in the field
of rational numbers $\Q$. In other words, the corresponding locus 
$\cC_\Q\subset\bP^2(\Q)$ is vacuous. (See \cite[p.~85]{Cas}.) 
\ssk

\begin{theo}\label{T-Nag} \it 
Let $\bF\subset\C$ be any subfield of the complex numbers, and let 
$\cC\subset\bP^2(\C)$ be an irreducible 
cubic curve, defined by a homogeneous equation  
\hbox{$\Phi=0$} with coefficients in $\bF$.
Then $\cC$ is $\bF$-projectively equivalent to a curve in the standard 
normal form $(\ref{E-snf})$ if and only if the set of non-singular 
 points in $\cC_\bF$ contains a flex point.
\end{theo}

\begin{rem}\label{R-nag} 
This is a much easier variant of  Trygve Nagell's \hbox{Theorem,}
which can be stated as follows:
\begin{quote} {\sl Given 
a smooth complex cubic $\cC$ which is defined over $\bF$, and given an
arbitrary point $\p\in\cC_\bF$, there is an $\bF$-birational equivalence
 between $\cC$ and some  curve in standard normal form which takes $\p$ to the
 flex point at infinity.}
\end{quote}
 See \cite{Nag}, as well as \cite[p. 34]{Cas}
which implicitly includes a brief proof of the above
 Theorem \ref{T-Nag}. For further discussion, see Remark \ref{R-birat} below.
\end{rem}

\begin{proof}[Proof of Theorem \ref{T-Nag}] Let $\cC$ be a curve in the 
normal form (\ref{E-snf}). Along the ``line at infinity'' with equation $z=0$,
 the equation $\Phi=0$ reduces to $x^3=0$,
yielding the single point $(0:1:0)$, counted with multiplicity three.
Thus $(0:1:0)$ is a flex point (non-singular since 
$\partial\Phi/\partial z \ne 0$), and the line $z=0$ is the tangent line at
this flex point.

Conversely, given any irreducible $\cC$ 
which is defined over $\bF$ and any flex point $\p\in\cC_\bF$,
we can put the curve into normal form in four steps, as follows.
\ssk

{\bf Step 1.} Choose an $\bF$-linear change of coordinates which maps $\p$
to the point $(0:1:0)$ and maps the tangent line at $\p$ to the line $z=0$.
It is then easy to check that the image of $\cC_\bF$ will have defining 
equation of the form $$\Phi(x,y,z)~ =~x^3+z\Psi(x,y,z)~,$$ where $\Psi$ 
is homogeneous quadratic with coefficients in $\bF$. Note that the coefficient
of $y^2z$ in $\Phi$ must be non-zero. In fact
it is easy to check that $$\Phi_x(0:1:0)=\Phi_y(0:1:0)=0$$
and that $\Phi_z(0,1,0)$ is equal to the coefficient of $y^2z$.
Since our 
flex point is assumed to be non-singular, this coefficient must be non-zero.\ssk

{\bf Step 2.} If we make a linear change of coordinates,
 replacing  $x$ by $\alpha x$
and $y$ by $\beta y$, and also replace $\Phi$ by $\Phi/\alpha^3$,
then the equation will take the form $\widehat\Phi=0$ where
$$\widehat\Phi(x,y,z)~=~ x^3 + z\Psi(\alpha x, \beta y, z)/\alpha^3~.$$
Thus the coefficient of $y^2z$ is now multiplied by $\beta^2/\alpha^3$.
Now choose $\alpha$ and $\beta$ so that the coefficient
of $y^2z$ will be $-1$. (As one example, there is a unique choice
with $\alpha=\beta$.) Working in affine coordinates with $z=1$, 
this means that our  $\widehat\Phi$ will take the form
$$ -y^2+x^3+px^2+ qx +r + y(sx+t)~,$$
with coefficients $p,\, q,\, r,\, s,\, t\, \in\, \bF$.\ssk

{\bf Step 3.} To get rid of the $y$ terms on the right, simply
 replace $y$ by $y+(sx+t)/2$. This will yield a function of the form
$$ -y^2+x^3+p'x^2+q'x+r'~.$$

{\bf Step 4.} To eliminate the  $x^2$ term, replace $x$ by $x-p'/3$.
Our function will then be in the required form $-y^2+x^3+ax+b$.\end{proof}
\medskip

\begin{lem}\label{L-sing} Using the normal form $(\ref{E-snf})$ with
$a,\,b\in\bF\subset\C$, the curve 
$\cC$ is singular if and only if the equation $x^3+ax+b=0$ has
a double root, which is necessarily in the subfield $\bF$, or if and only
if the discriminant $ -(4a^3+27b^2)$ is zero.
\end{lem}
\smallskip

\begin{proof}
Over the complex numbers, there is always an essentially unique factorization
$$x^3+ax+b~=~(x-r_1)(x-r_2)(x-r_3)~.$$
Suppose that $(x:y:z)$ is a singular point of $\cC$. Since $\Phi_z(0:1:0)=-1$,
the unique point of $\cC$ on the line $z=0$ is certainly non-singular,
so it will suffice to work in affine coordinates with $z=1$. 
Every point with $y\ne 0$ is non-singular since
$\Phi_y(x,y,1)=-2y\ne 0$. Thus it only remains to consider the three
points $(r_j:0:1)$ on the line $y=0$. For example
$\Phi_x(r_1,0,1)=(r_1-r_2)(r_1-r_3)$, so that the point $(r_1:0:1)\in\cC$
is singular if and only if $r_1$ is a double root.

Next we need to check that a double root necessarily
belongs to the subfield $\bF\subset\C$. But the sum of the roots is zero,
so if $r$ is a double root, then the third root is $-2r$. It follows easily
that $a=-3r^2$ and $b=2r^3$, so that either $a=b=r=0\in \bF$ or else
$r=-3b/2a\in \bF$.

Finally, we apply the classical discriminant identity
$$ \prod_{i<j} (r_i-r_j)^2~=~ -( 4a^3+27 b^2).$$
(See for example \cite{BM}.)  This proves Lemma \ref{L-sing}.
\end{proof}
\medskip

\begin{lem}\label{L-invar}
A projective change of coordinates $$(x:y:z)\mapsto (X:Y:Z)$$ which fixes the
flex point $(0:1:0)$ will transform a curve
in the standard normal form $y^2=x^3+ax+b$ to a curve $Y^2=X^3+AX+B$
in the same normal form
if and only if this change of coordinates has the form
\begin{equation}\label{E-ch-coord}
 X= t^2x~,\quad Y=t^3y~,\quad{\it and~hence}
\quad A=t^4a\quad{\it and}\quad B=t^6b\end{equation}
for some non-zero $t\in \bF$.
\end{lem}
\smallskip

\begin{proof}  If we make the substitutions (\ref{E-ch-coord}) in the
equation $$Y^2=X^3+AX+B\,,$$ then we obtain the original equation $y^2=x^3+ax+b$
multiplied by $t^6$. To show that this is the only permissible change
 of coordinates, we proceed as follows.

Since the line $z=0$ is  tangent  to our curve at the
marked flex point, it must certainly map onto itself under any projective
transformation
 which preserves this point and its tangent direction. Thus it will
suffice to work in affine coordinates, with $z=1$. 
The most general linear transformation then has the form
$$ X=(\alpha x+\beta y)+ \xi~,\quad Y=(\gamma x+\delta y) +\eta~, $$
with $\alpha,\,\beta,\,\gamma,\,\delta,\,\xi,\,\eta\,\in\,\bF$,
and with $\alpha\delta-\beta\gamma\ne 0$. 
Substituting these values into $X$ and $Y$, the equation $~Y^2=X^3+AX+B$
should reduce to a constant multiple of $~y^2=x^3+ax+b~$ for suitably chosen 
$A$ and $B$. Here the coefficient $\beta$ must be zero so that there is
no $x^2y$ term in the expansion, and $\gamma=0$ so that there is
no $xy$ term. Similarly $\xi=0$ so that there is no $x^2$
term, and $\eta=0$ so that there is no $y$ term. Thus $X=\alpha x$ and
 $Y=\delta y$. Finally, we must have $\delta^2=\alpha^3$ so that the
coefficients of $y^2$ and $x^3$ will be equal. Thus, setting 
$t=\delta/\alpha$, we have $t^2=\delta^2/\alpha^2=\alpha$ and
$t^3=\delta^3/\alpha^3=\delta$. Thus we obtain
$$ t^6 y^2~=~ t^6 x^3 + t^2 A x + B~.$$
Dividing by $t^6$, the required equations
$A=t^4a$ and $B= t^6b$ now follow.
\end{proof}
\medskip

\begin{coro}\label{C-classif}
Every smooth complex cubic curve $\cC$
can be reduced to the standard form $(\ref{E-snf})$
by a projective transformation; and two such curves are projectively
equivalent if and only if they share the same value for the invariant
\begin{equation}\label{E-J}
J(\cC)~=~\frac{4a^3}{4a^3+27 b^2}~\in\C~.\end{equation}
Here any value $J(\cC)\in\C$ can occur.
\end{coro}

\begin{proof} This follows directly from Theorem \ref{T-Nag} and Lemma
\ref{L-invar}.  There are three places where the restriction to the
complex case is necessary. First, according to Theorem \ref{T-9flex} every
smooth cubic curve has a flex point. Second, according to 
 Corollary \ref{C-aut}, there is an automorphism which carries any flex point to
any other flex point, so that it doesn't matter which flex point we choose
for the reduction to normal form. Third, since every complex number has
a complex square root, it follows easily  that we can use the transformation
$~ a\mapsto A=t^4a,~b\mapsto B=t^6b~$
for suitably chosen $t$
to convert the pair of coefficients $a,\,b$ to $A,\,B$,
whenever  the ratio $(a^3:b^2)\in \bP^1(\C)$ is equal to  $(A^3:B^2)$.

However, this ratio $(a^3:b^2)$ is a bit awkward to work with,
since either $a$ or $b$ may be
zero, and since the ratio $(-3^2:2^2)$ occurs only for singular curves.
The equivalent invariant (\ref{E-J}) is much more convenient since it takes
all possible finite values for smooth curves, and is infinite only
for singular curves by Lemma \ref{L-sing}.
Further details of the proof of Corollary \ref{C-classif} are straightforward.
\end{proof}
\medskip

\begin{rem}\label{R-Jsubf} 
If the curve $\cC$ is defined over a subfield $\bF\subset\C$, then
the invariant $J(\cC)$ must belong to $\bF$.
 In fact, since the flex points are defined by algebraic
equations with coefficients in $\bF$, they are are contained in\footnote
{Suppose for example that $(x:y:1)\in\bP^2(\C)$ is a flex point. Then the 
field $\bF''$ obtained from $\bF$ by adjoining $x$ and
 $y$ must be a finite extension of $\bF$. For otherwise  the inclusion map
$\bF\to\C$ would extend to infinitely many different embeddings of $\bF''$ 
into $\C$, leading to infinitely many flex points. The required $\bF'$ is now
just the splitting field of $\bF''$ over $\bF$.}
 $\bP^2(\bF')$ for some finite Galois extension $\bF'\supset\bF$; 
hence $J(\cC)\in\bF'$. But $J(\cC)$ is invariant
under all automorphisms of $\bF'$ over $\bF$, so it must belong to $\bF$.
\end{rem}
\medskip

We can give this invariant $J\in\C$ a more geometric interpretation as follows.
Recall that a curve in standard normal form is uniquely determined by the
three roots $r_j$, which are distinct if and only if the curve is non-singular.
We will use the word ``\textbf{\textit{triangle}}'' 
as a convenient term for an unordered
set consisting of three distinct points in the complex plane.
({\bf Caution:} The three points are allowed to lie in a straight line.)
\medskip

 \begin{definition} \label{D-shape}
 We will say that two subsets of the complex plane have the 
same \textbf{\textit{shape}} if there is a complex affine
automorphism $x\mapsto px+q$ which takes one to the other.
\end{definition}\smallskip

\begin{prop}\label{P-tri} \rm
For cubic curves of the form $y^2=f(x)$, where $f(x)$ is a cubic polynomial
with distinct roots $\{r_j\}$, the shape of the triangle formed by the
three roots is a complete invariant for projective equivalence.
\end{prop}
\smallskip

In particular, this is true for curves in the normal form $y^2=x^3+ax+b$.
Since $J(\cC)$ is also a complete invariant for
projective equivalence, it follows that this ``shape'' is uniquely
determined by the complex number $J$.
\smallskip

\begin{proof}[Proof of Proposition \ref{P-tri}] It is not difficult
to put a curve of the form\break
 $y^2=f(x)~$ into the standard form by an affine change of the $\,x\,$ 
variable. The conclusion then follows easily from Lemma \ref{L-invar}.
\end{proof}
\medskip

\begin{rem}\label{R-tri}
(See Figure \ref{F-jpic} for some typical examples.) Note that:
\begin{quote}
$\quad J=0~~\Leftrightarrow~~ a=0~~\Leftrightarrow\quad$ the triangle
 is equilateral, and

$\quad J=1~~\Leftrightarrow~~ b=0~~\Leftrightarrow~~$ one vertex is the
 midpoint of the other two.
\end{quote}
\noindent  For real values of $J$, the triangle is isosceles if $J<1$,
but the three vertices lie on a straight line if $J\ge 1$. If we label the
three edge
lengths $|r_j-r_k|$ as $e_1\le e_2\le e_3$, then $J\not\in\R$ if and only if 
$e_1<e_2<e_3$ and $e_3\neq e_1+e_2$. In fact the corresponding edges lie in 
positive (or negative)  cyclic order around the
triangle according as $J$ lies in the upper (or lower) half-plane. For 
a sequence of curves, $|J|$ tends to infinity if and only if the ratio
$e_3/e_1$ tends to infinity.
\end{rem}

\begin{figure}[t!]
\centerline{\includegraphics[width=4.5in]{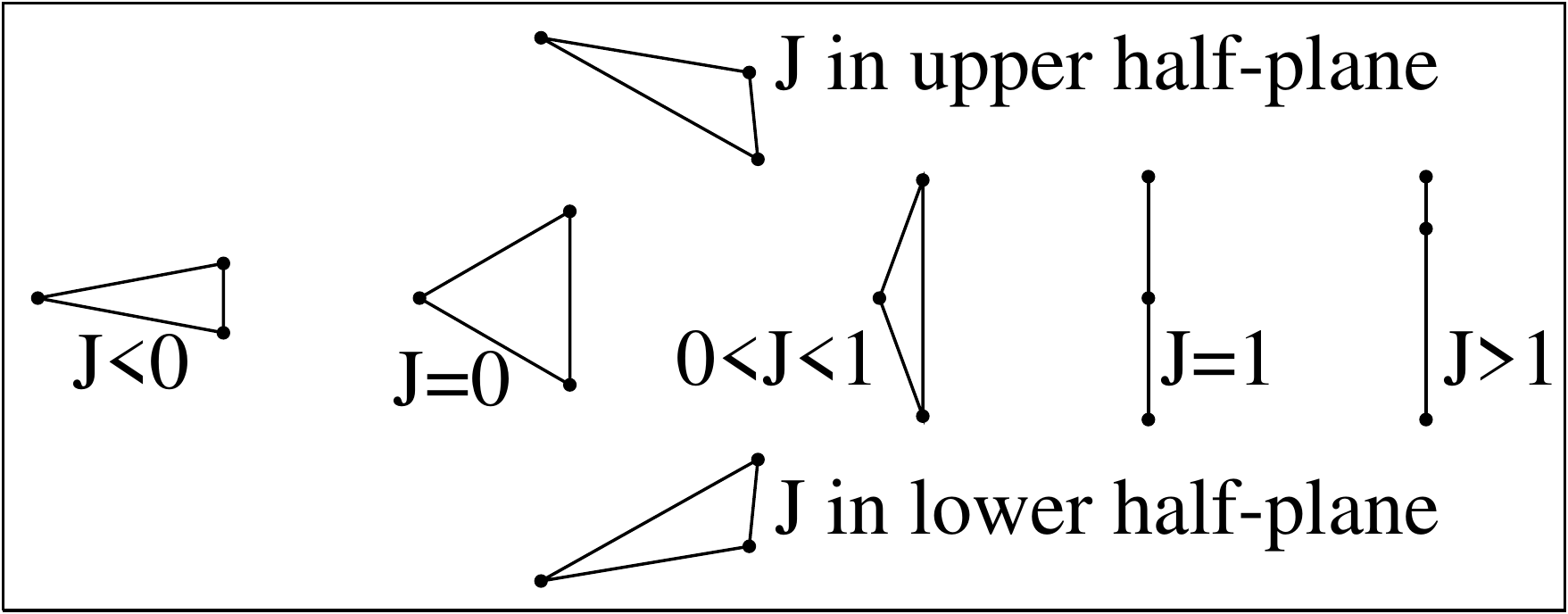}}
\mycaption{\label{F-jpic} \sf  The $J$-invariant describes the shape of the
$($possibly degenerate$)$ triangle in $\C$  with vertices $r_1,\,r_2,\, r_3$.}
\end{figure}

\bsk

We can now give a precise description of the automorphism group\break
(Compare Definition \ref{D-aut}).

\begin{coro}\label{C-aut2}  
The automorphism group of any smooth complex cubic curve can be described by
a split exact sequence
$$ 1~~\to~~N(\bP^2,\,\cC)~~\to~~\Aut(\bP^2,\,\cC)~~\to~\Aut(\bP^2,\,\cC,\,\p_0)~~\to~~1~.$$
Here:\vspace{-.4cm}
\begin{quote}
\begin{description}
\item[$\p_0~$] can be any one of the nine flex points,
\item[$\Aut(\bP^2,\,\cC,\,\p_0)$] is the subgroup of $\Aut(\bP^2,\,\cC)$ 
consisting of all automorphisms which fix the point $\p_0$, and
\item[$N(\bP^2,\,\cC)\cong \Z/3\oplus\Z/3$] is the normal 
subgroup consisting of all automorphisms which have no fixed point on $\cC$,
together with the identity automorphism.
\end{description}
\end{quote}
\noindent Furthermore, $N$ is a maximal abelian subgroup,
 and acts simply transitively on the set of nine flex
points. The subgroup $\Aut(\bP^2,\,\cC,\,\p_0)$ is cyclic of order:
\vspace{-.4cm}
\begin{quote}
\begin{description}
\item[six~] if $J(\cC)=0$,

\item[four~] if $J(\cC)=1$, but

\item[two~] in all other cases.
\end{description}
\end{quote}\end{coro}
\smallskip

\noindent Thus the full automorphism group has order either 54, 36, or 18.
Note that the exceptional cases $J=0$ and $J=1$ are precisely the cases where
the ``triangle'' of Figure \ref{F-jpic} has rotational symmetry of order
three or two.
\medskip

{\bf Proof of Corollary \ref{C-aut2}.} The subgroup $N$ is normal since the
property of acting without fixed points is invariant under inner automorphism.
Since we know by Lemma~\ref{L-hess-aut} and  Theorem~\ref{T-hnf} that 
$N$ acts simply transitively on the flex points, it follows that
 any automorphism can be expressed uniquely as the composition of an element
 of $N$ with an element of $\Aut(\bP^2,\,\cC,\,\p_0)$. To compute the 
latter group, using the standard normal form,
take $\p_0$ to be the point $(0:1:0)$. According to Lemma~\ref{L-invar},
an automorphism fixing this point must have the form
$$(x:y:z)~\mapsto ~(t^2x: t^3y: z)~,\quad{\rm with}\quad
a\mapsto t^4a\quad{\rm and}\quad  b\mapsto t^6b~. $$
Thus when $a=0$ the coefficient $t$ can be any sixth root of unity, and when
$b=0$ it can be any fourth root of unity, but otherwise it can only be $\pm 1$.
(Expressed invariantly, the cyclic subgroup $\Aut(\bP,\,\cC,\,\p_0)$
 acts on the tangent space to $\cC$ at $\p_0$ by multiplication by a 
corresponding root of unity.) The conclusion follows.\qed
\medskip

\subsection*{From Hesse to Standard Normal Form.}
Since every cubic equation in Hesse normal form can be converted to one in
the standard normal form (see the proof of Theorem \ref{T-Nag}), it 
 follows  that the invariant
$J\big(\cC(k)\big)\in\C$ can be computed as a function of the Hesse parameter
 $k$. In fact, since we can always multiply the parameter $k$ by a cube root
of unity without changing the projective equivalence class,
simply by dividing one of the coordinates by this root of unity,
it follows that $J(\cC)$ can be computed as a function of $k^3\in\C\ssm\{1\}$.
The computation is straightforward (if somewhat tedious), and yields the
following result in our notation:
\begin{equation}\label{E-Fr}
  J\big(\cC(k)\big)~=~\left(\frac{k(k^3+8)}{4 (k^3-1)}\right)^3~.
\end{equation}
(Compare \cite{Fr}, as well as \cite[Prop. 2.3]{PP}.)
It follows from this expression that the invariant $J\big(\cC(k)\big)$ 
tends to infinity whenever $k^3$ tend to either infinity or $+1$. It also
follows from this expression (or from Remark \ref{R-Jsubf}) that 
 every rational value of $k$ corresponds to a rational value of $J$, or to
$J=\infty$. (However, an irrational value of $k$ may correspond to a 
rational $J$. For example $k=1\pm\sqrt 3$ yields $J=1$.)

One noteworthy property is the following. Let $\bet:\widehat\C\stackrel{\cong}
{\longrightarrow}\widehat\C$ be the M\"obius involution
\begin{equation}\label{E-eta}
\bet(k)~=~\frac{k+2}{k-1}~,\qquad{\rm with}\qquad \bet\circ\bet(k)=k~.
\end{equation}
It will be convenient to use the abbreviated
expression $J(k)$ for $J\big(\cC(k)\big)$.
\smallskip

\begin{lem}\label{L-k2k} This function $J(k)=J\big(\cC(k)\big)$ satisfies
the identity
$$ J\big(\bet(k)\big)~=~J(k)\qquad{\it for~all}\qquad k\,\in\,
\widehat\C=\C\cup\{\infty\}~.$$
\end{lem}
\smallskip

\begin{figure}[ht!]
\centerline{\includegraphics[width=2.8in]{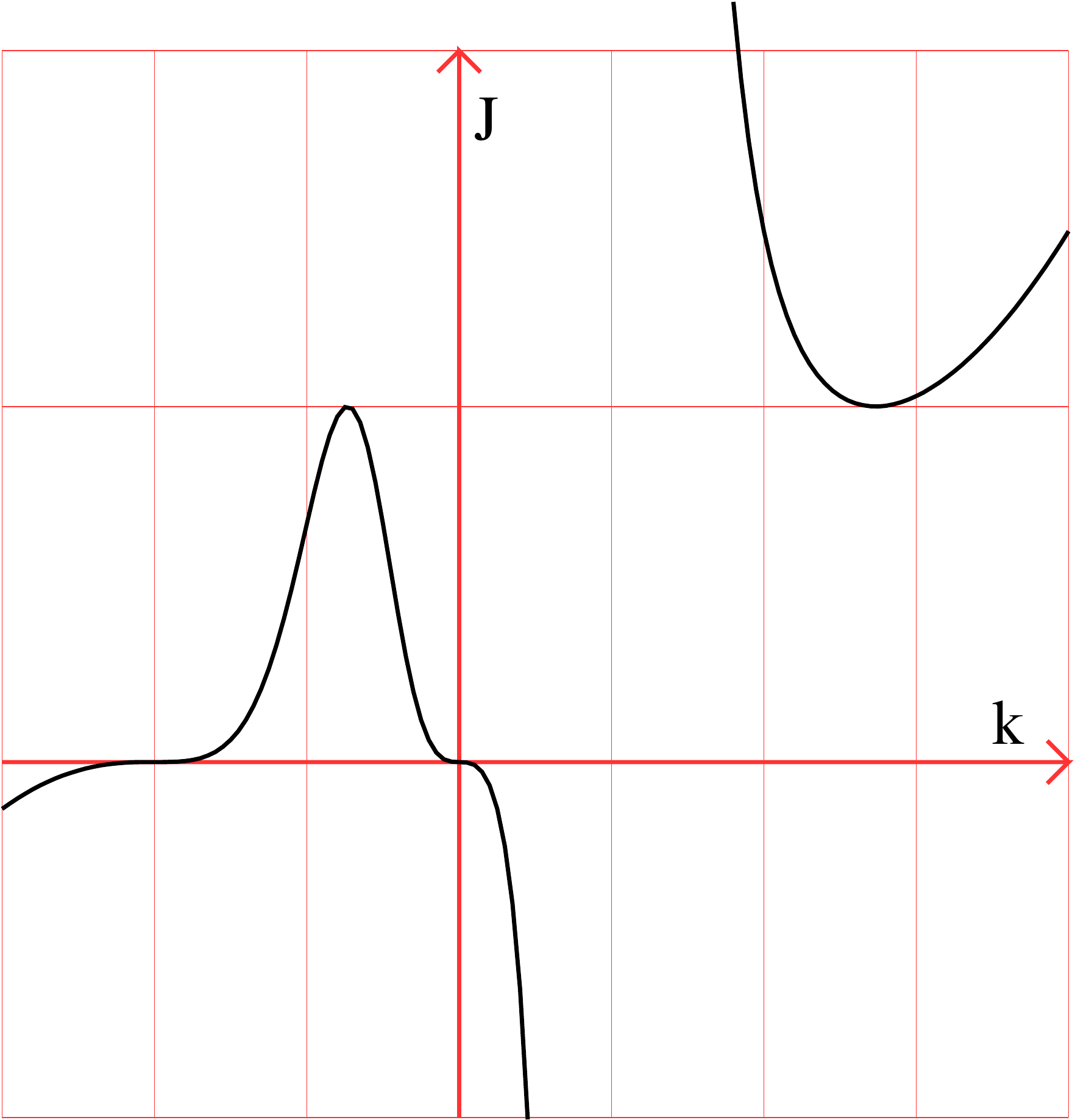}}
\mycaption{\label{F-kJ} \sf Graph of the map $~k\mapsto J(\cC(k))~$
of equation (\ref{E-Fr}) 
for real values of $k$, with $k\in[-3,\,4]$ and $J\in[-1,\,2]$.
Note that every line $J={\rm constant}$
intersects the graph in exactly two points. As examples, for $J=0$ we have
$k=0$ or $k=-2$, while for $J=1$ we have $k=1\pm\sqrt{3}$. 
 The graph is divided into two connected  components: The component 
with $k<1$ represents curves $\cC(k)_\R$ which are connected, while the
component with $k>1$ represents curves with two connected components.}
\end{figure}
\medskip

In particular, it follows that the graph, shown in Figure~$\ref{F-kJ}$,
 is invariant under the involution 
$$(k,J)~~\longleftrightarrow~~ \big(\bet(k),\,J\big)~,$$
which maps the region $k<1$ to itself with fixed point $(1-\sqrt 3,\,1)$, and
the region $k>1$ to itself with fixed point $(1+\sqrt 3,\,1)$. (Note that both
fixed points lie along the line $J=1$.)

\bigskip

{\bf First Proof.} The equation $J\big(\bet(k)\big)=J(k)$ is an identity
between two rational functions  of degree twelve, which can be 
verified by direct computation.\qed
\bigskip

However, this argument gives no clue as to how to construct an actual
projective equivalence   between $\cC(k)$ and $\cC\big(\bet(k)\big)$. 
That can be remedied as follows.\medskip

{\bf Second Proof.} Let
$$\begin{matrix} X&=& x\,+\,y\,+\,z\\ Y&=&x+\gamma y +\bbar\gamma z\\
Z&=&x+\bbar\gamma y+\gamma z\\\end{matrix}$$
where $\gamma=e^{2\pi i/3}$. Then it is not hard to check that 
$$ X^3+Y^3+Z^3~=~ 3\Big( x^3+y^3+ z^3~+~6\,x\,y\,z\Big)~, $$
and that
$$ X\,Y\,Z~=~ x^3+y^3+z^3~-~ 3\,x\,y\,z~.$$
Setting $k=(x^3+y^3+z^3)/(3\,x\,y\,z)$, it follows easily that
$$\frac{X^3+Y^3+Z^3}{3\,X\,Y\,Z}~=~ \frac{k+2}{k-1}~,$$
and the conclusion follows.\qed
\bigskip

\begin{theo}\label{k-sym}
Let $\Gamma$ denote the group of M\"obius  transformations generated by
the involution $\bet$ and the rotation
$$\brh(k)~=~\gamma\,k~.$$
Then $\Gamma$ is equal to the twelve element
 \textbf{\textit{tetrahedral group}},  consisting of all M\"obius
transformations from the Riemann sphere to itself which map the four
point set $\{1,\,\gamma,\,\bbar\gamma,\,\infty\}$ to itself. Furthermore:

\begin{quote}
{\bf(1)} Two Hesse curves $\cC(k)$ and $\cC(k')$ are projectively equivalent
if and only if $~k'=\bmu(k)~$ for some $\bmu\in\Gamma$.
\end{quote} 

\begin{quote}
{\bf(2)} The function $k\mapsto J(k)$ can be computed as
$$ J(k)~=~\frac{1}{64}\prod_{\bmu\in\Gamma}\, \bmu(k)~.$$
\end{quote}
\end{theo}

{\bf Proof.} (Compare \cite{AD}.) Clearly $\bet: 1\leftrightarrow\infty$ under
 the involution $\bet$, and it is not hard to check that 
$\bet:\gamma\leftrightarrow \bbar\gamma$. It then follows easily that $\Gamma$
 can be identified with the group consisting of all even permutations
of these four symbols. 

To prove statement {\bf(1)}, note that the quotient
$\widehat\C/\Gamma$ is a Riemann surface, necessarily isomorphic to 
$\widehat\C$. Since the map $J:\widehat\C\to\widehat\C$ clearly has the property
that $J\circ\bmu=J$ for every $\bmu\in\Gamma$, it follows that $J$ can be
expressed as a composition
$$
\xymatrix@R=3.8pc@C=3.8pc@M=0.2pc{\widehat\C \ar[r]\ar[dr]_J  & \widehat\C/\Gamma \ar[d]^h\\ &  \widehat\C ~.}$$

Since both $J$ and the projection $\widehat\C\to\widehat\C/\Gamma$ have
degree twelve, it follows that the holomorphic map $h$ has degree one,
 and hence is a conformal diffeomorphism. Since two curves $\cC(k)$ and 
$\cC(k')$
are projectively equivalent if and only if $J(k)=J(k')$, it follows that they
are projectively equivalent if and only if $k$ and $k'$
have the same orbit under $\Gamma$.

To prove {\bf(2)}, note that the function $k\mapsto\prod_\bmu\, \bmu(k)$ is
also a rational map of degree twelve which is invariant under precomposition
with any $\bmu\in\Gamma$. Furthermore, this function maps zero to zero
and infinity to infinity, so it must be some constant multiple of $J$. 
To compute the precise constant, it is enough to understand one more example.
 
It is not hard to check\footnote{In particular, note that 
$\brh^{-1}\!\circ\bet\circ\brh(1+\sqrt{3})=1-\sqrt{3}$. It is noteworthy
that to pass between the two real points $1\pm\sqrt{3}$
on the locus $J=1$ we need to
make use of a M\"obius transformation with complex coefficients.}
 that the orbit of $1+\sqrt 3$ consists of the following six points,
 each counted twice since $1+\sqrt 3$ is a fixed point of $\bet$.
$$ 1+\sqrt{3},~~ (1+\sqrt{3})\gamma,~~ (1+\sqrt{3})\bbar\gamma, ~~
 1-\sqrt{3},~~ (1-\sqrt{3})\gamma,~~ (1-\sqrt{3})\bbar\gamma~.$$
Since $(1+\sqrt{3})(1-\sqrt{3})=-2$ and $\gamma\bbar\gamma=1$, the product
$ \prod_\bmu~ \bmu(1+\sqrt{3})$ is equal to $(-2)^6=64$. Comparing this
with $J(1+\sqrt{3})=1$, the conclusion follows.\qed
\bsk

\setcounter{lem}{0}
\section{Cubic Curves as Riemann Surfaces.}\label{S-conf}

\begin{theorem}\label{T-abel} \it Every smooth cubic curve is conformally
diffeomorphic to a flat torus of the form $\C/\Omega$ where $\Omega$ is a
\textbf{\textit{lattice}} $($that is, an additive subgroup generated by two 
complex numbers which
are linearly independent over $\R\,.)$ Here $\Omega$ is uniquely determined
up to multiplication by a complex constant. Hence the shape of $\Omega$ 
$($Definition \ref{D-shape}$)$ is a complete invariant for the conformal
diffeomorphism class of $\cC$.
\end{theorem}

This will be an immediate consequence of two lemmas. The first lemma
is based on methods introduced by Abel. (However Abel himself did not
work in projective space or discuss algebraic curves. He simply studied
integrals, for example of the form $\int dx/\sqrt{p(x)}$ where 
$p(x)$ is a polynomial.)

\begin{lem}\label{L-abel}
Every smooth cubic curve in $\bP^2(\C)$ possesses a \hbox{holomorphic} 1-form
$(=$ Abelian differential$)$ which is well defined and nowhere zero. 
This 1-form is unique up to multiplication by a non-zero complex constant.
\end{lem}

\begin{proof} We will use affine coordinates $(x:y:1)$,\, and take
 the curve in the standard  normal form $~y^2=x^3+ax+b~$, so that
\begin{equation}\label{E-dx&dy}
  2\,y\,dy~=~(3\,x^2+a)dx~. \end{equation}
Consider  the holomorphic 1-form\footnote
{Caution: This notation is not intended to suggest that $dw$ is the total
differential of a globally defined function. Of course we can integrate to
find a function which is locally well defined up to an additive constant;
 but the integral is not  well defined globally.}
$dw$ which is defined by
$$ dw~=~\frac{dx}{y}~\qquad{\rm whenever}\qquad y\ne 0~,$$
and by
$$ dw~=~\frac{2\,dy}{3\,x^2+a}\qquad{\rm whenever} \qquad 3\,x^2+a\ne 0~.
$$
(It follows from Equation~(\ref{E-dx&dy}) that these
two forms are equal when both are defined.)
The two denominators cannot both be zero since the equations 
$$~\Phi_x~=~\Phi_y~=~0~$$
would imply that  $\cC$ is singular.)
This form $dw$ is clearly well defined and non-zero at all points
of  $\cC$ which lie within the affine plane. Since the intersection of
$\cC$ with the line at infinity is the single flex point $(0:1:0)$,
it only remains to check what happens near this point.
To do this, we will work with alternative
affine plane in which $y=1$, setting 
$$\qquad x=X/Z~~~{\rm and}~~~ y=1/Z\quad{\rm so~ that}
\qquad(x:y:1)~=~(X:1:Z)~.$$
Using the equation
\begin{equation}\label{eq-cstr}
\Phi(X,1,Z)~=~ -Z+X^3+aXZ^2+bZ^3~=~0~,
\end{equation}
we see that $~\Phi_X(0:1:0)=0~$ and $~\Phi_Z(0:1:0)=-1$,
so that $X$ can be used as a local uniformizing parameter on $\cC$.
In fact,  we can
express $Z$ locally as a function of $X$ of the form  
$Z = cX^n + O(X^{n+1})$, with $n\ge 2$ since $Z=0$ is the tangent line. 
Substituting
this expression for $Z$ in the right hand side of the equation
$Z=X^3 + a X Z^2 +b Z^3$  it follows easily that $c=1$ and
$n=3$, so that
$$ Z = X^3 + O(X^5) ~,\quad{\rm and}\quad dZ=\big(3X^2+O(X^4)\big)dX~.$$
Now using the equation
$$ dw~=~\frac{dx}{y}~=~
\frac{d(X/Z)}{1/Z}~=~\frac{Z\,dX\,-\,X\,dZ}{Z} $$
it follows that
$$dw~=~ \big(-2\,+\,O(X^2)\big)\,dX
~.$$
Thus the holomorphic 1-form  $dw$ is smooth and non-zero, even at
the flex point $(0:1:0)$. Since any other holomorphic 1-form can be obtained
by multiplying $dw$ by a holomorphic function from
$\cC$ to $\C$, 
which is necessarily constant since $\cC$ is compact,  
this proves Lemma \ref{L-abel}.
\end{proof}
\medskip

\begin{lem}\label{L-cov}
Let $\cC$ be any compact Riemann surface which admits a nowhere zero
holomorphic 1-form $\eta$. Then the set of  integrals
$\oint_{\,\Lambda} \eta\in\C$,
where $\Lambda$ varies over all smooth closed loops in $\cC$,
forms a lattice $\Omega\subset\C$, and  $\cC$ is conformally  diffeomorphic
to the quotient Riemann surface $\C/\Omega$. 
\end{lem}
\smallskip

\begin{proof} (Compare \cite[p.~84]{Don}.)
Choose a base point $\p_0\in\cC$. Then the universal covering space 
$\widetilde\cC$ can be described as the set of 
all pairs $\big(\p, \{P\}\big)$ 
where $\p$ can be any point of $\cC$ and $\{P\}$ is any homotopy class of
smooth paths from $\p_0$ to $\p$. Given any such pair, 
we can integrate along 
any $P\in\{P\}$ to obtain a complex number $w=\int_P\eta\in\C$
which does not depend on the choice of $P$ within its homotopy class.
In other words, we have a well defined mapping
\begin{equation}\label{E-t}
\big(\p\,,\{P\}\big)~\mapsto w=\int_P\eta\qquad{\rm from}\quad 
\widetilde\cC\quad{\rm to}\quad\C~.
\end{equation}
Further, the total differential $dw$ of this function $w$
is just the 1-form $\eta$, lifted to the universal covering.

Using the flat Riemannian
metric $|dw|^2$, we see that this function (\ref{E-t}) is a conformal
isometry from $\widetilde\cC$ onto the complex numbers. In fact the inverse
map from $\C$ to $\widetilde\cC$ sends each straight line from the origin
in $\C$ to a corresponding geodesic in $\widetilde\cC$.

Now suppose that we have two different paths $P_1$ and $P_2$ from $\p_0$
to $\p$, yielding two complex numbers $w_1$ and $w_2$. Then the difference can
be expressed as
$$ w_1-w_2~=~\int_\Lambda \eta~, $$
where $\Lambda$ is the closed loop obtained by following $P_1$ from $\p_0$
to $\p$, and then following $P_2$ back to $\p_0$. Conversely, given any
closed loop $\Lambda$ from $\p_0$ to itself, we can first follow $P_1$
 and then follow $\Lambda$ to obtain a new path $P_2$ from $\p_0$ to $\p$.
This proves that two points in $\widetilde\cC$ map to the same point of $\cC$
if and only if the difference between their images in $\C$ differ by an element
of the additive group $\Omega\subset\C$.

Since the map from $\widetilde\cC$ to $\cC$ is a local diffeomorphism, it
 follows that $\Omega$ must be a discrete additive subgroup: that is, it cannot
contain non-zero elements arbitrarily close to zero. Furthermore, since
the quotient $\C/\Omega\cong\cC$ is compact, $\Omega$ must contain two
linearly independent elements. This proves Lemma \ref{L-cov}; and 
Theorem \ref{T-abel} then follows easily.
\end{proof}
\medskip

The converse assertion, that every flat torus ${\mathbb T}=\C/\Omega$ 
is conformally diffeomorphic
to a smooth cubic curve, is due to Weierstrass (\cite{W1}, \cite{W2}),
and arose from his study of doubly periodic functions. 
Since this result is widely known (see for example \cite[Sec.2]{La}),
we will give only a brief summary.
\smallskip

For any lattice $\Omega\subset\C$ the Weierstrass $\wp$-function is the unique
holomorphic $\Omega$-periodic map from $\C\ssm\Omega$ to $\C$ which has a pole
of the form
$$ \wp(w)~=~ 1/w^2 + o(1) \qquad {\rm as}\qquad w\to 0~.$$
This satisfies a differential equation of the form
$$ \big(\wp'(w)\big)^2~= ~ 4\wp(w)^3 \,-\,g_2\,\wp(w)\,-\,g_3~,$$
where the complex constants $g_2$ and $g_3$ can be computed from the lattice
$\Omega$. In fact,
$$ g_2~=~ 60\sum_{\omega\ne 0}\frac{1}{\omega^4}\quad{\rm and}\quad
g_3~=~ 140\sum_{\omega\ne 0}\frac{1}{\omega^6}~,$$
where $\omega$ ranges over all non-zero lattice elements. (Compare 
\cite[\hbox{p.83-84}]{Ser}.) Setting
$$ (X:Y:Z)~=~(\wp(w):\wp'(w):1)~,$$
this yields a conformal diffeomorphism\footnote{To check differentiability
near $w=0$, we can set $\wp(w)=w^{-2}+\epsilon(w)$ where $\epsilon(w)$ is
holomorphic. Then $w$ maps to
$$\big(\wp:\wp':1\big)=\Big(w^{-2}+\epsilon(w): -2w^{-3}+\epsilon'(w):1\Big)
=\Big(w+w^3\epsilon(w): -2+w^3\epsilon'(w): w^3\Big)~,$$
clearly yielding a local conformal diffeomorphism.}
from the torus ${\mathbb T}=
\C/\Omega$ onto the cubic curve
$$ Y^2~=~ 4\,X^3-g_2X-g_3~.$$ 
This can easily be transformed into our standard normal form by setting 
$$ Y=2y\quad{\rm and}\quad X=x~,\qquad{\rm with}\qquad 
 g_2~=~-4a\quad{\rm and}\quad g_3~=~-4b~.$$
Felix Klein showed that
 the $J$-invariant can be computed as a holomorphic 
function of the lattice parameter $\tau$, where 
$\Omega=\Z\oplus\tau\Z$ with $\Im(\tau)>0$. See for example \cite[p.90]{Ser}.
\smallskip

\begin{coro}\label{C-aut3}
The group $\Aut(\cC)\cong\Aut({\mathbb T})$ of conformal automorphisms
of the curve $\cC\cong{\mathbb T}$ can be described by a split exact sequence
$$ 1~\to~ N(\mathbb T)~\to~\Aut({\mathbb T})~\to~\Aut({\mathbb T},\,0)~\to~1~,$$
where the normal subgroup $N(\mathbb T)\cong{\mathbb T}$ of automorphisms
without fixed point can be identified with the group of translations
 of $\mathbb T\cong\C/\Omega$, and where the finite cyclic subgroup
$\Aut({\mathbb T},\,0)\cong\Aut(\cC,\,\p_0)$ is naturally isomorphic to the
group $\Aut(\bP^2,\,\cC,\,\p_0)$ of Corollary~$\ref{C-aut2}$.
\end{coro}
\smallskip

 \begin{proof} Note first that the derivative of any conformal automorphism of
 $\T$ is a holomorphic function from the compact surface
 $\T$ to $\C$, and hence must be constant.
Hence any automorphism must be linear. But the only linear maps without
fixed points are translations. Further details are easily supplied.
\end{proof}
\medskip

\begin{coro}\label{C-PE&CD} \it Two smooth cubic curves are projectively
equivalent if and only if they are conformally  diffeomorphic.
A given conformal diffeomorphism extends to an automorphism 
 of $\bP^2(\C)$ if and only if it maps flex points to flex points. 
\end{coro}
 \smallskip

\begin{proof} If the two curves are projectively equivalent, then they
are certainly conformally diffeomorphic. Conversely, if they are 
conformally diffeomorphic, then it follows from
the discussion above that they have a common $J$-invariant, 
hence by Corollary \ref{C-classif} they are 
projectively equivalent. Any projective equivalence
between two curves certainly sends flex points to flex points. Conversely,
given any conformal equivalence from $\cC_1$ to $\cC_2$ which sends flex points
to flex points, we can choose a projective equivalence from $\cC_2$ to $\cC_1$.
The composition will then be a conformal automorphism of $\cC_1$ which sends
flex points to flex points. Using Corollary \ref{C-aut3}, it is then not
difficult to check 
that this composition is a projective equivalence from $\cC_1$ to itself, and
the conclusion follows. 
\end{proof}

\begin{rem}[\bf Birational Maps]\label{R-birat}   
In place of conformal 
diffeomorphisms, we could equally well work with the purely algebraic 
concept of birational maps. Let $\f=(f_1:f_2:f_3)$ be a non-zero triple of 
homogeneous polynomial maps $\C^3\to\C$ of the same degree,
well defined up to simultaneous multiplication by a non-zero complex constant.
Let ${\mathcal I}({\bf f})\subset\bP^2(\C)$ be the locus of common zeros:
$f_1=f_2=f_3=0$. Then the function
$~\f:\bP^2(\C)\ssm{\mathcal I}(\f)\to \bP^2(\C)$
defined by the formula
$$ (x:y:z) \mapsto \big(f_1(x,y,z):f_2(x,y,z):f_3(x,y,z)\big) $$
is called a \textbf{\textit{rational map}} of $\bP^2(\C)$. 

It will be convenient to use the phrase \textbf{\textit{almost everywhere}}
to mean ``{\sl except on a finite subset}''. 
If $\cC$ is a curve in projective space such that the intersection
$\cC\cap \I(\f)$ is finite, and if the image
$\f\big(\cC\ssm \I(\f)\big)$ is contained 
 in a curve $\cC'$, then we obtain an almost everywhere defined map from $\cC$
to $\cC'$. Two such almost everywhere defined maps will be called 
\textbf{\textit
{equivalent}} if they agree almost everywhere. An equivalence class of such
maps will be called a \textbf{\textit{rational map}} from $\cC$ to $\cC'$.
If a rational map has an inverse, so that the composition is the identity map
almost everywhere, then it is called a \textbf{\textit{birational map}} 
from $\cC$ to $\cC'$. 
Given a birational map, there are finite subsets $S\subset\cC$ 
and $S'\subset\cC'$ so that $\cC\ssm S$ maps to $\cC'\ssm S'$ by a conformal 
diffeomorphism. Since the ``singularities'' (as the word is used in complex
 function theory) at the points of $S$ are clearly 
removable, it follows 
 that every birational map between smooth curves extends to a 
uniquely defined conformal diffeomorphism. In particular, the birational map
can be assigned a unique well defined value at every point. 

If we combine this discussion with  Nagell's Theorem, as described in Remark 
\ref{R-nag}, then we obtain the following.

\begin{coro}\label{C-birat} Every conformal diffeomorphism between smooth
cubic curves  is birational. Hence
the group of all birational maps from  a smooth cubic $\cC$ to itself
can be identified with the Lie group $\Aut(\cC)$ consisting of
all conformal automorphisms of $\cC$.
\end{coro}

\begin{proof} First note that every projective equivalence is birational.
From the discussion above, we see that every birational map is a conformal
diffeomorphism.

Let $f:\cC\to\cC'$ be a conformal diffeomorphism between
smooth cubic curves, and let $\p\in\cC$ be a flex point.
By Nagell's Theorem there is a smooth curve  $\cC''$ and a birational map
$g:\cC'\to\cC''$ taking $f(\p)$ to a flex point $\p''\in\cC''$.
By Corollary \ref{C-PE&CD}
there exists a projective equivalence $h:\cC''\to \cC$, and
by Corollary \ref{C-aut2}
we may choose $h$ so that it maps $\p''$ to $\p$. Since the composition
$h\circ g\circ f$ maps $\p$ to itself, it
 follows by Corollary~\ref{C-aut2} that this composition is a 
projective equivalence. Since $g$, $h$, and $h\circ g\circ f$ are all
birational equivalences, it follows that $f$ is also.
\end{proof}

\end{rem}
\medskip

\begin{rem}\label{R-Vor}
One curious invariant of the lattice $\Omega$ is the tiling
of the complex plane by \textbf{\textit{Voronoi cells}}
 $V_\omega=\omega+ V_0$, where $\omega$ varies over $\Omega$, and where
$V_0=V_0(\Omega)$
is the compact convex polygon consisting of all $z\in\C$ such that
$$ |z|~=~\min_{\omega\in\Omega}|z-\omega|~.$$ 
This polygon $V_0$ is a canonically defined fundamental domain for the
\hbox{additive}
action of $\Omega$ on $\C$; and is a 
complete invariant for $\Omega$, since $\Omega$ is the additive group
generated by the reflections of zero in the edges of $V_0$. The shape of $V_0$
is evidently a complete invariant for the conformal diffeomorphism class of
$\cC\cong{\mathbb T}$ (where two polygons centered at the origin 
 have the same ``shape''
if a complex linear automorphism maps one to the other).
 In particular, the group $\Aut({\mathbb T},\,0)$ can be identified with
the group of rotational  symmetries of $V_0$. This has order 6 if $V_0$ is a
regular hexagon (with $J=0$),  
order 4 if $V_0$ is a square (with $J=1$), and order 2 otherwise.
In most cases $V_0$ is a non-regular hexagon (as in Figure \ref{F-Vor}). 
However, it is a rectangle if $J$ is real with $J>1$.

\begin{figure}[ht!]
\centerline{\includegraphics[width=2.3in]{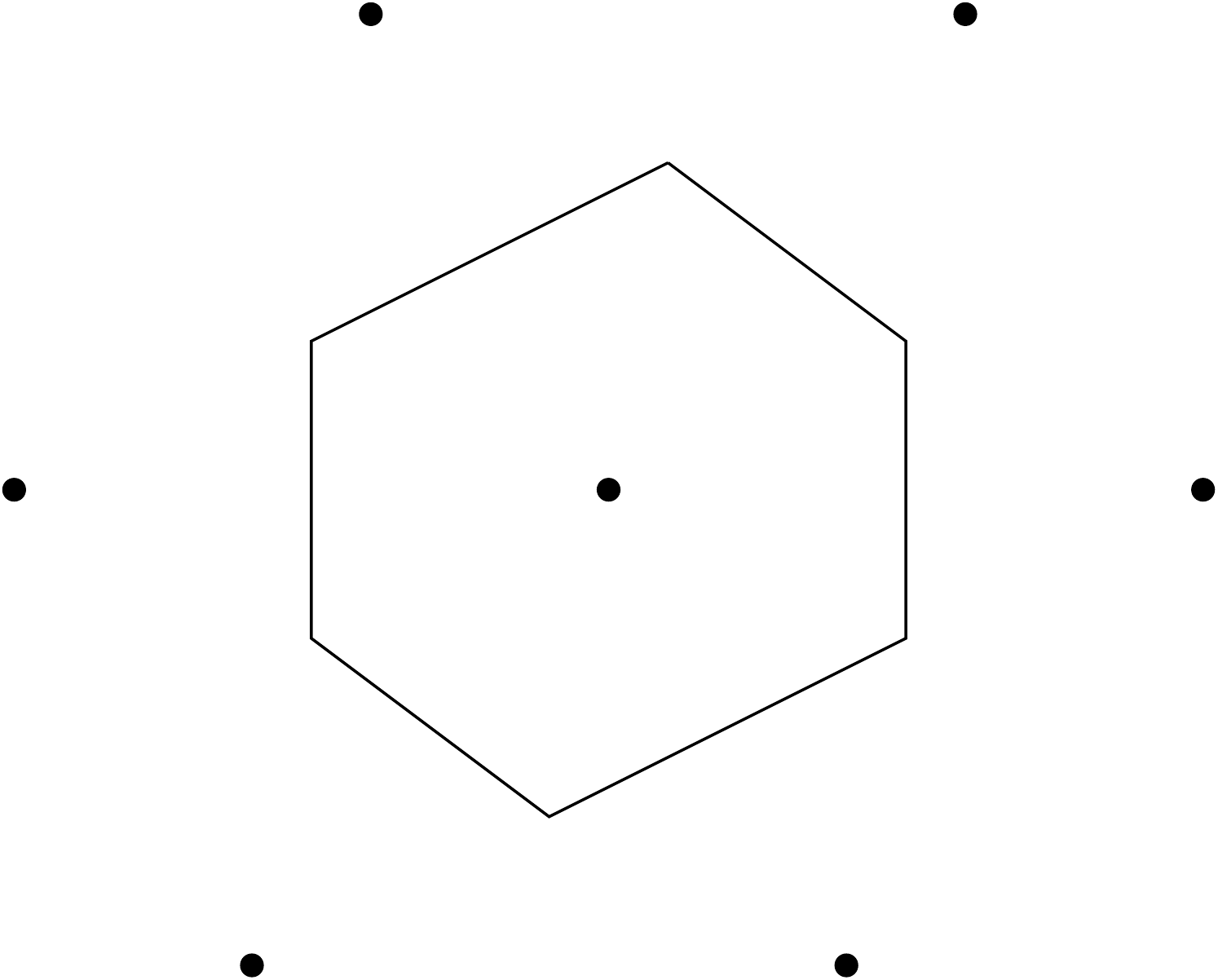}}
\smallskip
\mycaption{\label{F-Vor} \sf Voronoi hexagon for the lattice 
$\Z\oplus\tau\Z$ with $\tau=(3+4\,i)/5$.
The Voronoi polygon for any lattice has 180$^\circ$ rotational
symmetry. In this example, since the lattice has two generators of
equal length, it  also has an orientation reversing symmetry, which implies
that the $J$-invariant is real.}
\end{figure}

\end{rem}
\bigskip

\setcounter{lem}{0}
\section{The Chord-Tangent Map and Additive 
Group Structure.}\label{S-CT}

We first discuss the chord-tangent map.
Let $\cC\subset\bP^2(\C)$ be a smooth cubic curve.  Recall
that an arbitrary line $L\subset\bP^2$ intersects $\cC$ in exactly
 three points, counted with multiplicity. It will be convenient to
call an unordered list $(\p,\,\q,\,\r)$ of three (not necessarily
distinct) points of $\cC$ a \textbf{\textit{collinear triple}} if it can
be obtained in this way, indicating multiplicity by duplication.

\begin{definition}\label{D-ctm}
The correspondence $(\p,\,\q)\mapsto \r$, where $(\p,\q,\r)$ is any collinear 
triple, will be called the \textbf{\textit{chord-tangent map}} from
$\cC\times\cC$ to $\cC$, and will be denoted by
\begin{equation}\label{E-*}
 (\p,\,\q)~~\mapsto~~ \p*\q~.\end{equation}
Note that the equation $\p*\q=\r$ is invariant under any permutation of
$\p,\,\q,\,\r$, and simply means that $(\p,\,\q,\,\r)$ is a collinear triple.
\end{definition}

 For example, 
if $\p=\q\ne\r$, then the equation $\p*\p=\r$ means that 
$(\p,\,\p,\,\r)$ is a collinear triple, and hence that the tangent line to
 $\cC$ at $\p$ also intersects the curve $\cC$ at the point $\r$.

\begin{lem}
 \it For any smooth complex cubic $\cC$, 
this chord-tangent map $$(\p,\,\q)\to \p*\q$$
 is holomorphic as a map from $\cC\times\cC$ to $\cC$.
\end{lem}\ssk

\begin{proof}
It is first necessary to show that the line $L$
determined by two points $\p$ and $\q$ in $\cC$ depends holomorphically on the
pair $(\p,\,\q)$. This is clear if $\p\ne\q$, but we must also consider
the limiting case as $\p$ and $\q$ tend to a common limit. 
Using affine coordinates $(x,\,y,\,1)$, and assuming that the slope
$s$ is finite, so that $L$ is defined by an equation $y=sx+c$, it clearly
suffices to prove that $s$ depends holomorphically on $\p$ and $\q$
as $\p$ and $\q$ tend to a common point. Describing the curve locally
by a holomorphic function $y=f(x)$, the slope of the line between
$\big(x_1,\,f(x_1)\big)$ and $\big(x_2,\,f(x_2)\big)$ is defined by 
$$ s(x_1,\,x_2)~=~\begin{cases} \displaystyle{\frac{f(x_1)-f(x_2)}{x_1-x_2}} & {\rm if}
\quad x_1\ne x_2\,,\quad{\rm but}\\[2ex]
f'(x) &  {\rm if}\quad x_1=x_2=x~. \end{cases}
$$
A standard power series argument shows that $s$ is holomorphic as a
function of two variables.

Let $\Phi(x,y,1)=0$ be the defining equation for the affine curve.
Assuming that we have chosen coordinates so that the point $\r=\p*\q$
also belongs to the affine plane, 
the function $\Phi(x,y,1)$ restricted to the line $y=sx+c$ determined by $\p$
 and $\q$ can be expressed as a cubic polynomial 
$$\Phi|_L~=~c_0x^3+c_1x^2+c_2x+c_3\qquad{\rm with}\qquad c_0\ne 0~,$$ 
where the coefficients $c_j$ depend
holomorphically on $\p$ and $\q$. Factoring this polynomial as 
\hbox{$c_0(x-p)(x-q)(x-r)$}, we have $p+q+r= -c_1/c_0$. Therefore 
$r=-p-q-c_1/c_0$
also depends holomorphically on $\p$ and $\q$. Thus the $x$-coordinate
of the required point $\r=\p*\q\in L$ varies holomorphically, so $\r$ does
also.
\end{proof}
\smallskip

\begin{rem}\label{R-ct/F}
 As in \S\ref{S-class}, it is interesting to see what happens over an arbitrary
subfield ${\mathbb F}\subset\C$
Assuming that $\cC$ is defined by equations with coefficients in $\bF$,
recall that $\cC_\bF$ is defined to be the intersection $\cC\cap\bP^2(\bF)$. 
If $(\p,\,\q,\,\r)$
is a collinear triple for $\cC$, with $\p$ and $\q$ in $\cC_\bF$,
then it is  
not hard to check that $\r\in\cC_\bF$ also\footnote{If a polynomial equation 
has coefficients in ${\mathbb F}$, note that then sum of its roots is also in
 ${\mathbb F}$.}. Thus the chord-tangent map
$(\p,\,\q)\mapsto \p*\q$ is well defined as a map from $\cC_\bF\times\cC_\bF$
to $\cC_\bF$.

In the case that $\bF$ is the field $\Q$ of rational numbers, the map
$\p\mapsto\p*\p$ was used by Diophantus of Alexandria in the third century
to construct new points of $\cC_\Q$ out of known ones. (For examples,
see \cite[pp. 24--25]{Cas}.)
\end{rem}
\smallskip

\begin{figure}[ht!]
\centerline{\includegraphics[height=1.8in]{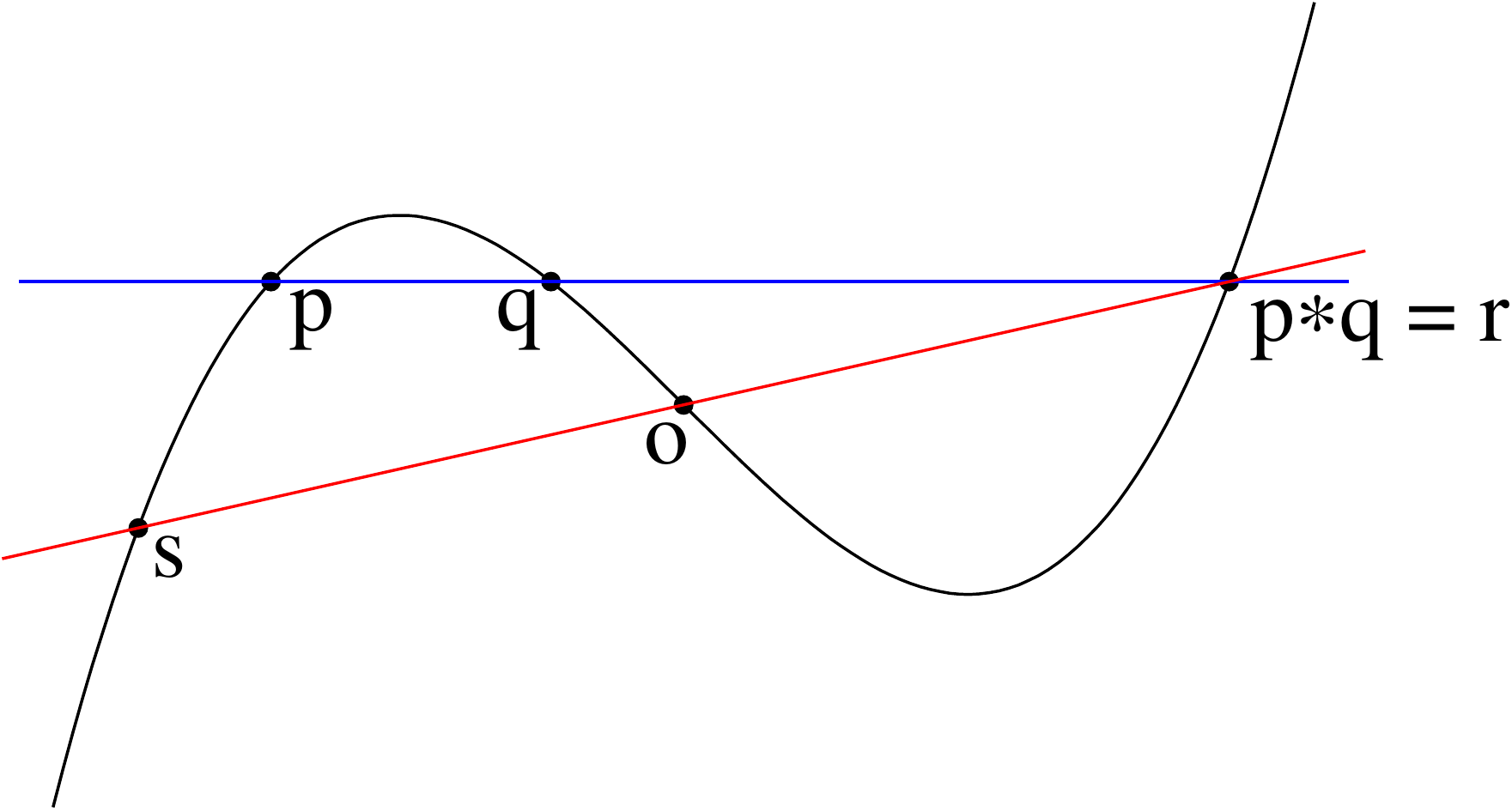}}
\mycaption{\label{F-sum} \sf Constructing the sum  $~\p+\q=\s$.}
\end{figure}
\smallskip

Next we will use the chord-tangent map to describe the additive group structure
of a smooth cubic curve.
\smallskip

\begin{lem}
 \it Let $\o$ be an arbitrarily chosen base 
point\footnote{ The term \textbf{\textit{elliptic curve}} 
is often reserved for a smooth cubic curve
together with  a specified base point.}
in the smooth cubic curve $\cC\subset\bP^2(\C)$.
Then $\cC$ admits one and only one additive group structure with the following
two properties:

\begin{itemize}
\item[\bf(1)] The base point $\o$ is the zero element, so that $\o+\p=\p$ for
any $~\p\in\cC$.

\item[\bf(2)] The triple $(\p,\,\q,\,\r)$ is collinear $($as defined above$)$
 if and only if the sum $~\p+\q+\r$ takes a constant value which
depends only on the choice of $\o$.
\end{itemize}
\end{lem}
\smallskip

{\bf Proof of uniqueness.} Assume that such a group structure exists.
For any $\p$ and $\q$, let $\r=\p*\q$ and let
${\bf s}=\r*\o$ as in Figure \ref{F-sum}, using the notation (\ref{E-*}).
 Then by Property {\bf(2)} we have the identity
$$ \p+\q+\r~=~\r+\o+{\bf s}~.$$
Canceling the $\r$'s and using Property {\bf(1)}, it follows that $\p+\q=\s$, 
or in other words
\begin{equation}\label{E-sum}
 \p+\q~=~(\p*\q)*\o~.
\end{equation}
This proves uniqueness.
\smallskip

\begin{rem} 
The constant $\p+\q+\r$ in Property {\bf (2)} is necessarily equal to $\o*\o$,
as we see by considering the collinear triple $\o,\, \o,\, \o*\o$. Similarly,
since $~\p,\,~ \o*\o,\,~ (\o*\o)*\p~$ forms a collinear triple,
we see that the additive inverse $-\p$ is equal to $(\o*\o)*\p$.

\end{rem}
\smallskip

{\bf Proof of existence.} Define the sum operation by the formula
(\ref{E-sum}), setting $\r=\p*\q$ and $\p+\q=\r*\o$
as illustrated by Figure~\ref{F-sum}.   Note the 
identity $(\p*\q)*\q=\p$ for all
$\p$ and $\q$. In particular, taking $\q=\o$, we have
$$\p+\o~=~(\p*\o)*\o~=~\p$$
for all $\p$. 
Thus $\o$ is indeed
a zero element for the sum operation. 

For any collinear triple$(\p,\,\q,\,\r)$, as in the diagram, we can compute
the sum
$$(\p+\q)+\r~=~\s+\r~=~(\s*\r)*\o~=\o*\o~,$$
which is constant, as required.

This sum operation is clearly commutative.
Over a general field, the proof of associativity is somewhat tricky.
(Compare \cite{Cas}.) However, in the complex case it is quite easy: First note
that for fixed $\q\ne\o$ the mapping $\p\mapsto\p+\q$ from $\cC$ to itself
has no fixed points. In fact, with $\r*\p=\q$ and $\r*(\p+\q)=\o$ as in
Figure \ref{F-sum}, the equation $\p=\p+\q$ would imply that $\q=\o$.

Now choose a conformal diffeomorphism 
$\psi:\cC\stackrel{\cong}{\longrightarrow}\bT$
to the appropriate torus $\bT=\C/\Omega$, normalized by the requirement that
$\psi(\o)=0$. Then translation by $\q\ne\o$ on $\cC$ corresponds to a fixed 
point free conformal diffeomorphism from $\bT$ to itself which maps zero to
$\psi(\q)$. But the only such isomorphism is the translation by $\psi(\q)$.
It follows easily
that the transformation $\psi$ is not only a conformal diffeomorphism
but also preserves the sum operation. Therefore the sum is associative;
 and $\psi$ is an isomorphism of additive groups. 
\qed
\bigskip

\begin{figure}[h!]
\centerline{\includegraphics[height=2.5in]{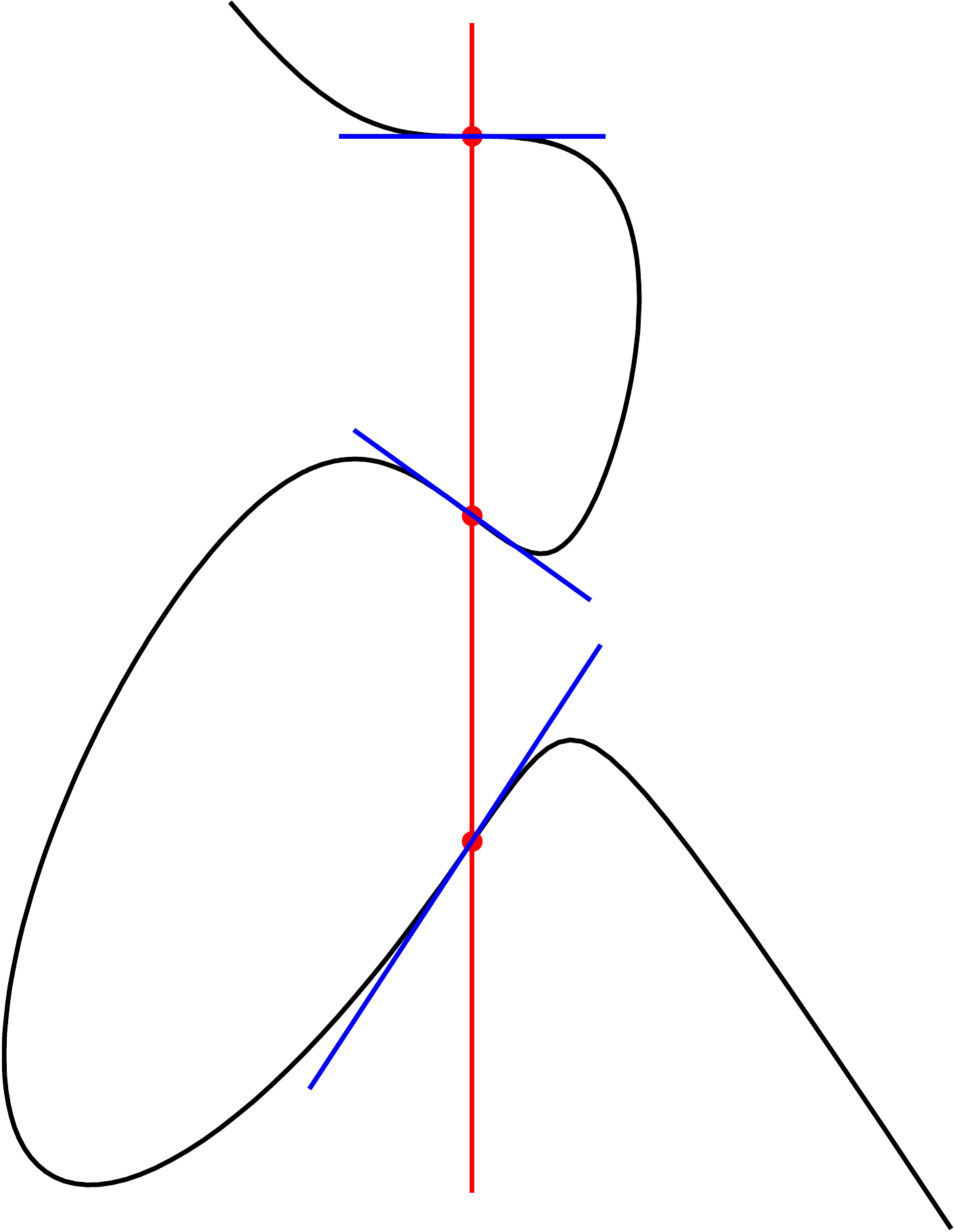}}
\mycaption{\sf The line between two distinct
 flex points always intersects $\cC$ in a
third flex point.  
\label{fig-flex}}
\end{figure}
\smallskip

\begin{figure}[ht!]
\centerline{\includegraphics[width=2.6in]{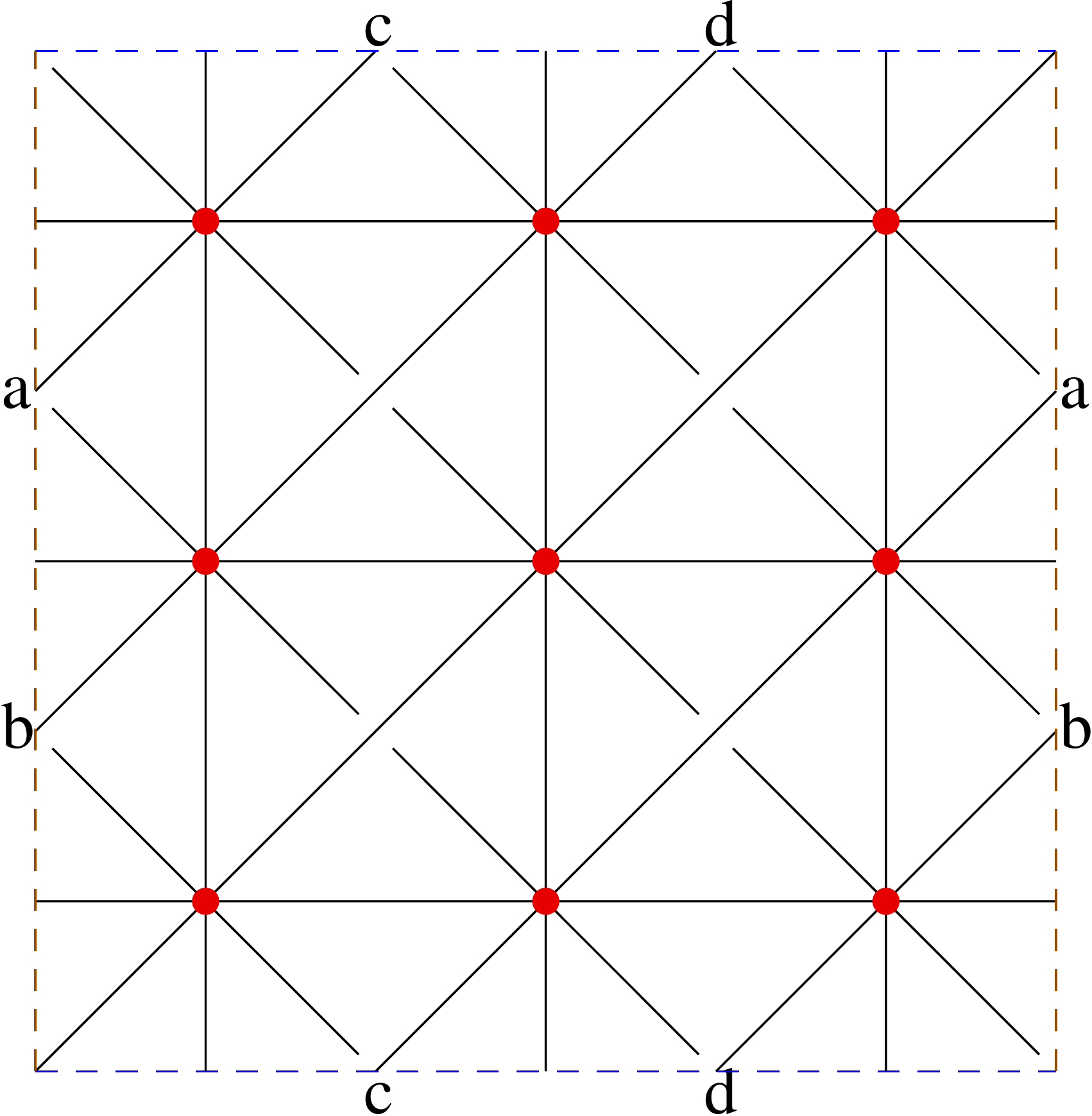}}
\mycaption{\sf \label{F-flexgrid}  A schematic picture of
the Hesse configuration consisting of nine flex points
together with the twelve lines joining them, placed on a square with opposite
sides identified. $($Compare \cite[Lehrsatz 12]{H2}, as well as
 $\cite{AD}.)$ This configuration has the nice property that
any two points determine a line and any two lines determine a point.
This configuration cannot be realized by straight lines in $\R^3$,
but can be more or less realized on a flat torus, as illustrated.}
\end{figure}

\begin{rem}\label{R-flexbase} 
If $\o \in \cC_{\mathbb F}$ for some subfield 
${\mathbb F}\subset {\mathbb C}$, then it follows that $\cC_{\mathbb F}$
is a subgroup of $\cC$.
This construction is particularly convenient when  $\cC_{\mathbb F}$ has a flex
point. In this case,  we can choose a flex point
as base point $\o$, so that $\o*\o=\o$, and so that
$\p+\q+\r~=~\o$ for any collinear triple. As an example, with this
choice the classical
``tangent process''  $\p\mapsto \p*\p$ is given by the formula
$$\p~~\mapsto~-2\,\p~.$$
One important consequence is that: {\sl the line joining any two distinct flex
points must contain a third flex point.}
(Compare Figure~\ref{fig-flex}.) With this choice of base point,
the flex points are precisely the elements of order three, satisfying
$\p+\p+\p=\o$ within the additive group. In the complex case, this additive
group of flex points has order nine, and hence, is isomorphic to
 $\Z/3\oplus\Z/3$.
\end{rem}
\medskip

\begin{rem}\label{R-Hconfig}
It follows easily that every smooth complex cubic contains a configuration
of nine flex points which joined by twelve lines, where every two
 points determine a line and every two lines determine a point.
This ``\textbf{\textit{Hesse configuration}}'' can never be
realized by real\footnote{Remember that three generic points on a complex line
lie on a real circle, not on a real line.}
 straight lines, even in a high dimensional real space. However,
it can almost be realized on a flat torus,   as 
illustrated schematically in Figure  \ref{F-flexgrid}.
\end{rem}

\bsk

\setcounter{lem}{0}
\section{Real Cubic Curves.}\label{S-R} 
This section is concerned with cubic curves $\cC_\R\subset\bP^2(\R)$
defined by equations $\Phi(x,y,z)=0$ with real coefficients.
We will describe the curve $\cC_\R$ as \textbf{\textit{smooth and irreducible}}
if the locus $\cC_\R$ itself contains no singular points and contains
no line.\footnote{Thus we do not allow examples such as $\Phi(x,y,z)
=x(x^2+y^2+z^2)$. In this example, the real locus is just a non-singular line
$x=0$; but the complex locus has singular points at $(0:\pm i:1)$ where the 
two irreducible components intersect.} 
{\it This is equivalent to the requirement that  the associated full complex
locus $\cC\subset\bP^2(\C)$ must have no singular points.} 
In fact, if $\cC$ has just
one singular point, then it must be invariant under the complex
conjugation map $(x:y:z)\leftrightarrow(\overline x:\overline y:\overline z)$,
and hence must belong to $\cC_\R$. If there are two
complex conjugate singular points, then the complex line joining them
must have intersection multiplicity at least two with each point, hence this
entire line must be contained in the curve $\cC$. Since this line is
self-conjugate, its intersection  with $\bP^2(\R)$ will be a line in $\cC_\R$.

The problem of classifying real cubic curves 
was studied already by Isaac 
Newton (but in the affine plane; see \cite{Ne} and compare \cite[p. 284]{BK}).
In general, the  projective classification of real curves is 
parallel to the complex classification, however there are important differences.
In looking at pictures of real cubic curves, 
it is important to remember that the real projective plane is a non-orientable
manifold, and that every real cubic curve has a non-orientable neighborhood,
which can never be completely pictured within an affine plane (Remark
\ref{R-top}).
\medskip

\begin{lem}\label{L-realflex}
Every smooth irreducible real cubic curve contains a flex point.
\end{lem}
\smallskip

\begin{proof}[Proof of Lemma \ref{L-realflex}]
Since the full complex curve $\cC$ is smooth, it has nine
flex points. The complex conjugation map from $\cC$ to itself, with fixed
point set $\cC_\R$, must permute these nine points. Since it is an involution,
 it must fix at least one of them.
\end{proof}
\medskip

Thus it follows from Theorem \ref{T-Nag} that we can put $\cC_\R$ into the
\hbox{standard} form $$y^2~=~x^3+ax+b$$ by a real projective transformation.
In particular, it follows that the invariant 
$~J(\cC_\R) =4a^3/\big(4a^3+27b^2\big)~$ is a real number.
\medskip

\begin{lem}\label{L-real-invar}
For each $J\in\R$ there are two essentially different smooth irreducible
real cubic curves. A complete invariant for smooth
real curves in this normal form, up to real projective equivalence,
 is provided by:
\begin{quote}
$\bullet$ this invariant $~J(\cC_\R)~$  together with

$\bullet$ the sign of $b$ if $b\ne 0$, or

$\bullet$ the sign of $a$ if $b=0$.
\end{quote}
\end{lem}
\smallskip

(Note that $a$ and $b$ cannot both be zero since $\cC_\R$ is smooth.)
\smallskip

\begin{proof} According to Lemma \ref{L-invar}, the only allowable
transformations replace the pair of coefficients $(a,\,b)$ by $(t^4a,\,t^6b)$
for some non-zero real number $t$. Since $t^4>0$ and $t^6>0$, the signs of
$a$ and $b$ are both invariants. However, if we are given both $b$ and $J$
then we can solve uniquely for $a^3$, provided that $b\ne 0$, so the sign of
$a$ is uniquely determined. The conclusion then follows easily. 
\end{proof}
\medskip

\begin{figure}[ht!]
\centerline{\includegraphics[width=1.5in]{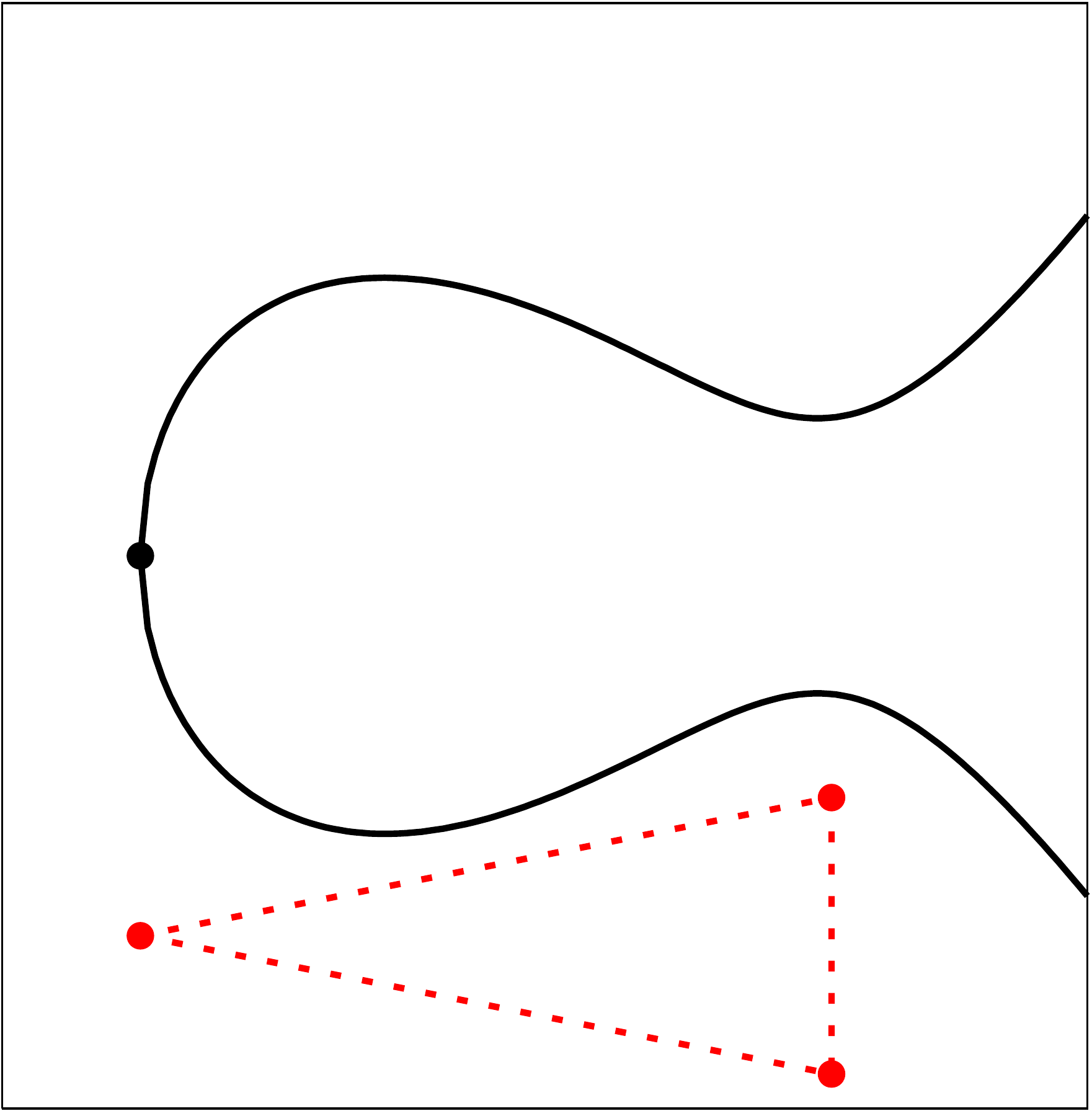}\qquad\qquad
\includegraphics[width=1.5in]{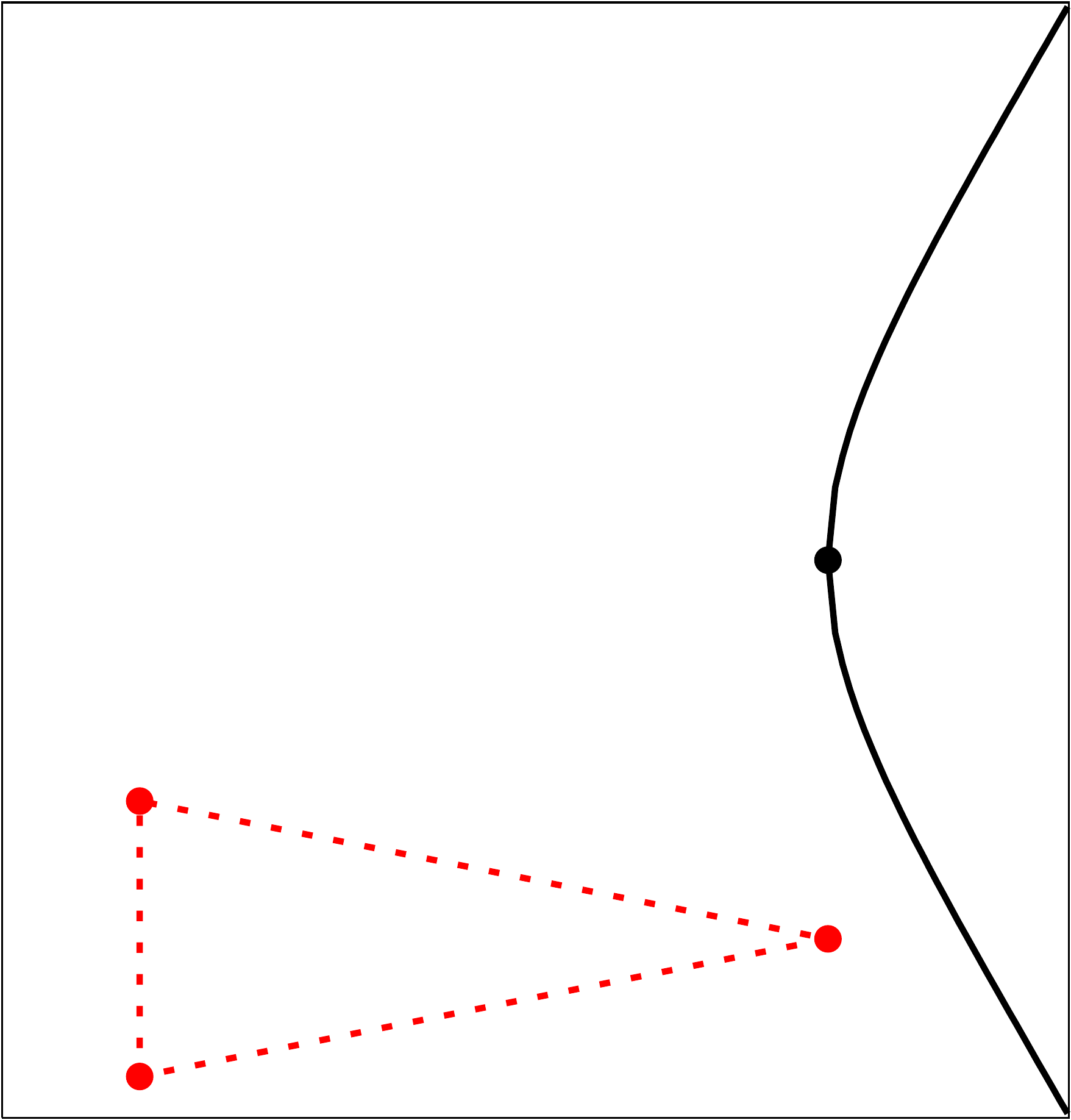}}\vspace{-.3cm}
$$ k\approx-3.91,~b>0,\qquad\qquad k\approx-0.58,~b<0,$$\smallskip

\centerline{\includegraphics[width=1.5in]{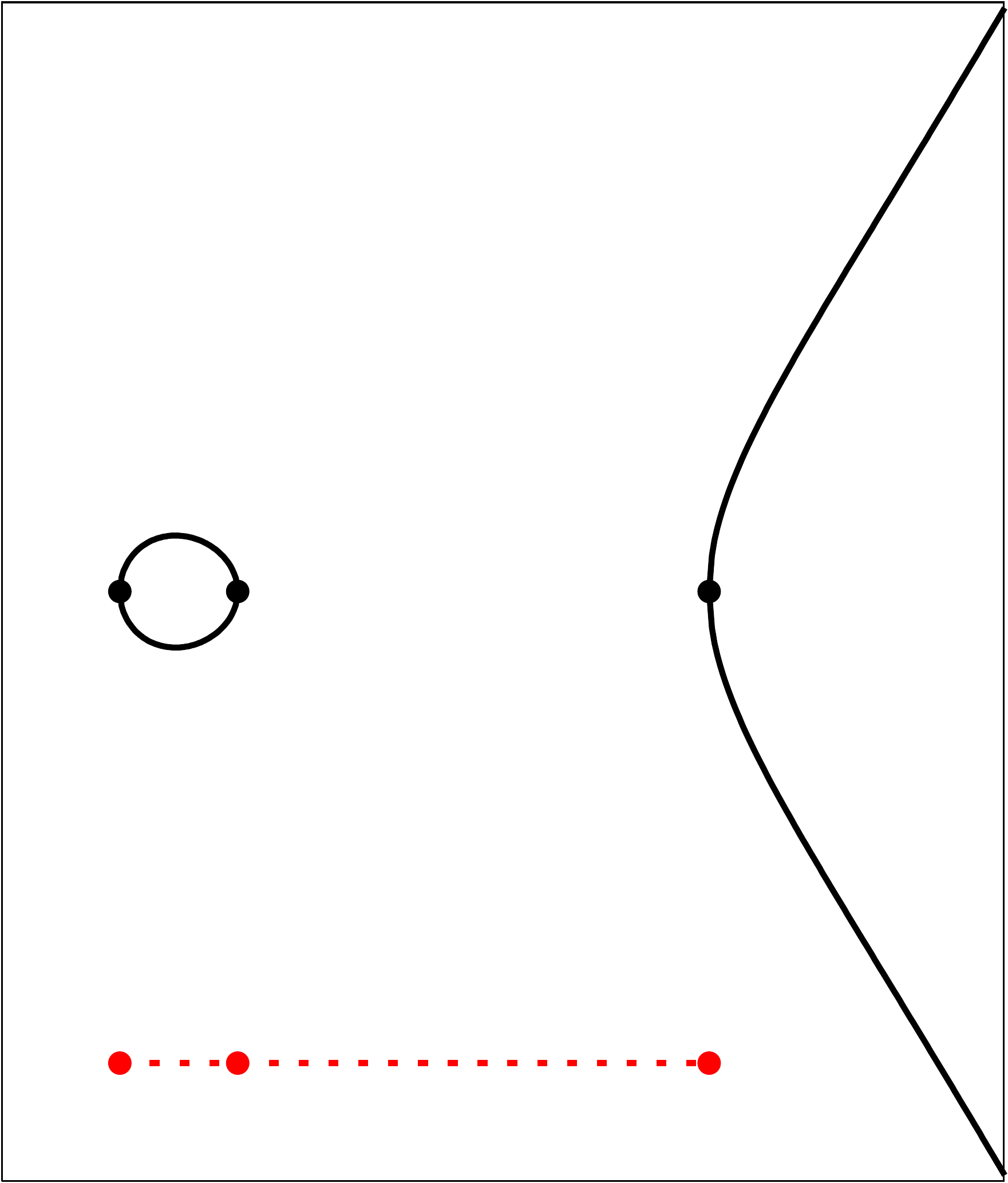}\qquad\qquad
\includegraphics[width=1.5in]{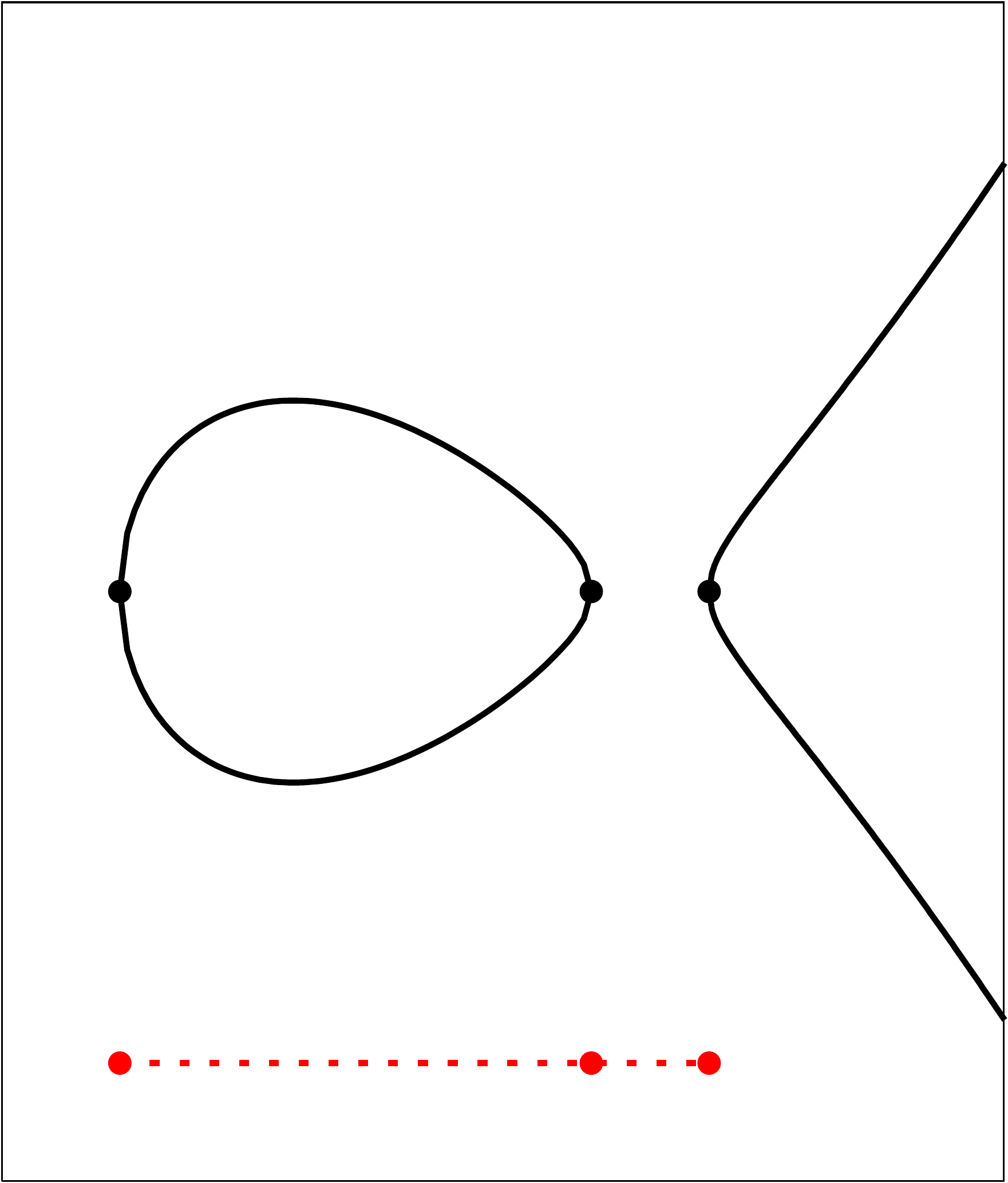}}\vspace{-.3cm}
$$ k\approx 1.63,~b<0,\qquad\qquad k\approx 5.75,~ b>0~.$$
\mycaption{\label{F-realwei} \sf Examples of pairs of distinct real curves
 in standard normal form which have the same $J$ invariant, giving the
corresponding value of the Hesse $k$ invariant. 
$($See Theorem $\ref{T-realH}.)$
The curve in the real $(x,y)$-plane
is shown in solid curves, and the corresponding
 triangle in the complex \hbox{$x$-plane} is
shown below in dotted lines. For the two top figures we have $J=-.583$,
and for  the bottom figures, $J=3.43$.}
\end{figure}
\smallskip

More geometrically, if the transformation
$$ x\mapsto t^2x\,, ~~y\mapsto t^3,~~ a\mapsto t^4a\,,~~ b\mapsto t^6b$$
is to change the sign of $b$ without changing $a$, then we must have $t^2=-1$.
Thus we must also change the sign of $x$. In particular, the associated  
triangle in the complex $x$-plane will be rotated by $180^\circ$.
But we we must also multiply $y$ by $\sqrt{-1}$, which
 makes a drastic change in the real curve.
 Compare Figures \ref{F-jpic} and \ref{F-realwei}.
Similarly, a change in the sign of $~a~$ corresponds to a $90^\circ$
rotation of the complex $x$-plane.\medskip

Now compare Figure \ref{F-kJ}. This graph shows that each real $J$
corresponds to two possible values of the Hesse parameter $k$ (although
the case $~J=1\Leftrightarrow b=0~$ seems quite different from the other
 cases). For $J\ne 1$ the two distinct real values of $k$ are related
by the involution\break $k\leftrightarrow\bet(k)$ of equation (\ref{E-eta}).
In fact, we have the following statement.
\medskip

\begin{theo}\label{T-realH}
Every smooth real cubic curve $\cC_\R$ is real projectively equivalent to the
real Hesse curve $\cC(k)_\R$ for one and only one real $k\ne 1$.
This curve $\cC(k)_\R$ is connected if $k<1$, and has two components if $k>1$.
\end{theo}

To begin the proof, note that $\cC_\R(k)$ is smooth if and only if 
\hbox{$k\ne 1$.} (Compare Lemma \ref{L-Hsing}.)

\begin{lem}\label{L+-b}
For $k\ne 1$, putting this curve into the  standard normal form 
\hbox{$y^2=x^3+ax+b$,} we have 
$b<0$ if and  only if $$1-\sqrt 3~<~k~<~1+\sqrt 3~,$$
and $b=0$ if and only if $k=1\pm\sqrt 3$, with $b>0$ otherwise.
\end{lem}

 \begin{proof}
Note first that $J=1$, or equivalently $b=0$, if and only if \break
\hbox{$k=1\pm\sqrt 3$.}  (Compare Figure \ref{F-kJ},  
together with the accompanying discussion.)
The two extremal points $k=1\pm\sqrt 3$,  together with the separating value
$k=1$, cut the real line into
four subintervals such that $J\ne 1\Leftrightarrow b\ne 0$ on each
subinterval. Thus it is enough to check one example on each subinterval,
as shown for example in Figure \ref{F-realwei}. 
\end{proof}

{\bf Note:} In the case $k<1$ with $\cC$ connected,
 a pair of test examples which is even easier to work 
with is the following: The Hesse curve $\cC(-2)_\R$ is projectively equivalent
 to the curve $y^2=x^3+x$ in standard form, while $\cC(0)_\R$ is projectively
 equivalent to  $y^2=x^3-x$. (These two examples, with
 $k\in(-\infty,\,1-\sqrt 3)$ and $k\in(1-\sqrt 3,\,1)$ respectively, both
 correspond to the case $J=0$.)
For $k>1$, a more geometric discussion will be given in Remark \ref{R-k>1}
 below.

\medskip

\begin{proof}[Proof of Theorem \ref{T-realH}] It follows easily from 
Lemma \ref{L+-b} that, for each \hbox{$J\in\R$,} the two distinct values of $k$
correspond to two real curves which are not real projectively equivalent
since they are distinguished by the sign of $b$ (if $J\ne 1$), or the sign
of $a$ if $J=1$.
Thus there is a one-to-one correspondence between real projective equivalence
classes and real parameters $k\ne 1$. 

Finally, since the number of connected
components cannot change as $k$ varies over 
either of the connected intervals $(-\infty,\,1)$
and $(1,\,+\infty)$, it is enough to count the number of components
for one example in each interval. 
\end{proof}
\medskip

\begin{coro}[\bf Flex Points]\label{C-3flex}
Every smooth real cubic curve $\cC_\R$ has exactly three flex points.
\end{coro}
\smallskip

\begin{proof} In Hesse normal form, the flex points are just the
``exceptional points'' listed in Equation (\ref{E-hess-flex}). Evidently
exactly three of these points are real, namely the three points
$(x:y:z)$ with $$x+y+z~=~xyz~=~0\,.$$ The conclusion follows.
\end{proof}
\medskip

\begin{rem}[\bf Topology]\label{R-top}
By definition, a simple closed curve in the real projective plane is 
\textbf{\textit{essential}} if it
generates the homology group $$H_1\big(\bP^2(\R);\, \Z\big)\cong\Z/2~.$$
Every essential simple closed curve has a neighborhood which is a M\"obius band;
while every inessential one bounds a topological disk.
As an example, every line in $\bP^2(\R)$ is essential.
Two simple closed curves with transverse intersections have an odd number of 
intersections if and only if both curves are essential. If we think of 
$\bP^2(\R)$ as a unit sphere with antipodal points identified, then an
essential curve is covered by a simple closed curve which cuts the sphere
into two antipodal pieces; while an inessential curve is covered by a pair of
simple closed curves which cut the sphere into three pieces.

It is not hard to see that every smooth irreducible real cubic $\cC_\R$
has a unique essential connected component, which
contains the three flex points. If there is a second component, then it
must be inessential.
\end{rem}
\medskip

\begin{coro}[\bf Automorphisms]\label{C-realaut}
The projective automorphism group $\Aut\big(\bP^2(\R),\,\cC_\R\big)$ 
 is non-abelian of order six and can be identified with the group of
 permutations of the three flex points. That is, every
 permutation of the flex points extends uniquely to a
projective automorphism of the pair $\big(\bP^2(\R),\,\cC_\R\big)$.
\end{coro}
\smallskip

\begin{proof}
 Using the Hesse normal form,
it follows easily that the permutations of the three coordinates yield
a group of six automorphisms, which can be identified with the group of six
permutations of the three flex points.
To finish the proof, we
 must show that any automorphism which fixes all three flex points is the
identity. However, any real automorphism clearly extends to a complex
automorphism, so we can apply Corollary \ref{C-aut}. Any automorphism
which fixes one flex point $\p_0$ acts on the curve by a rotation by a root
of unity around $\p_0$; but the only real roots of unity are $+1$, which
corresponds to the identity automorphism, and $-1$ which 
interchanges the other two flex points. The conclusion follows.
\end{proof}
\medskip

\begin{figure}[!ht]
\centerline{\includegraphics[width=2.5in]{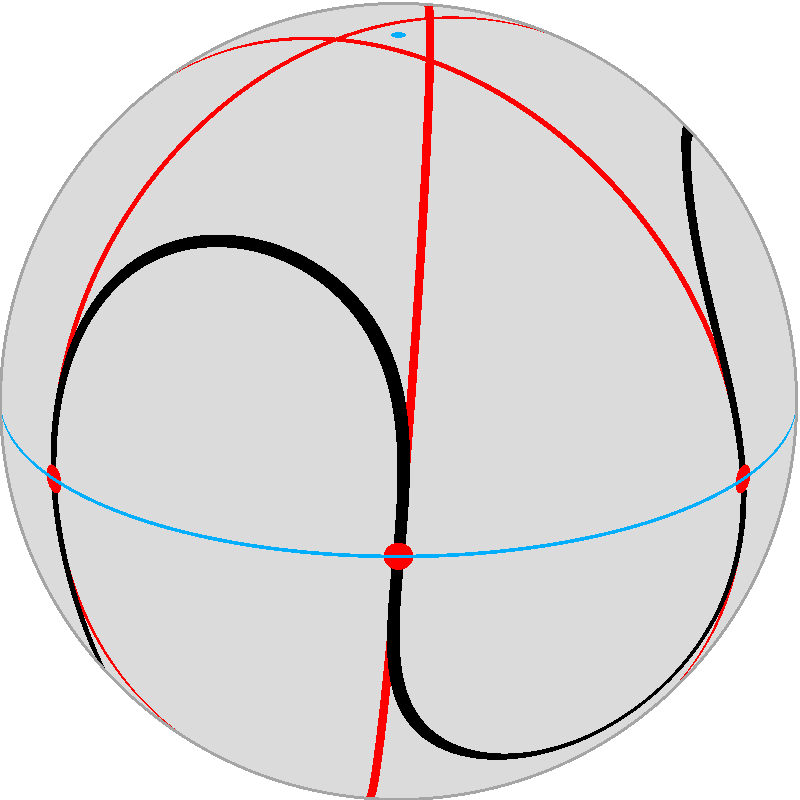}}
\mycaption{\label{F-sym}\sf Showing a typical real Hesse curve $\cC(-2.4)_\R$,
 with the projective plane
 $\bP^2(\R)$ represented as a sphere with opposite points 
identified.
The tangent lines at the three {\bf flex points} of this curve
 are also shown, as well
 as the center of symmetry (the north-south pole), and the line
through the three flex points (the equator).}
\end{figure}

\begin{figure}[!ht]
\centerline{\includegraphics[width=1.3in]{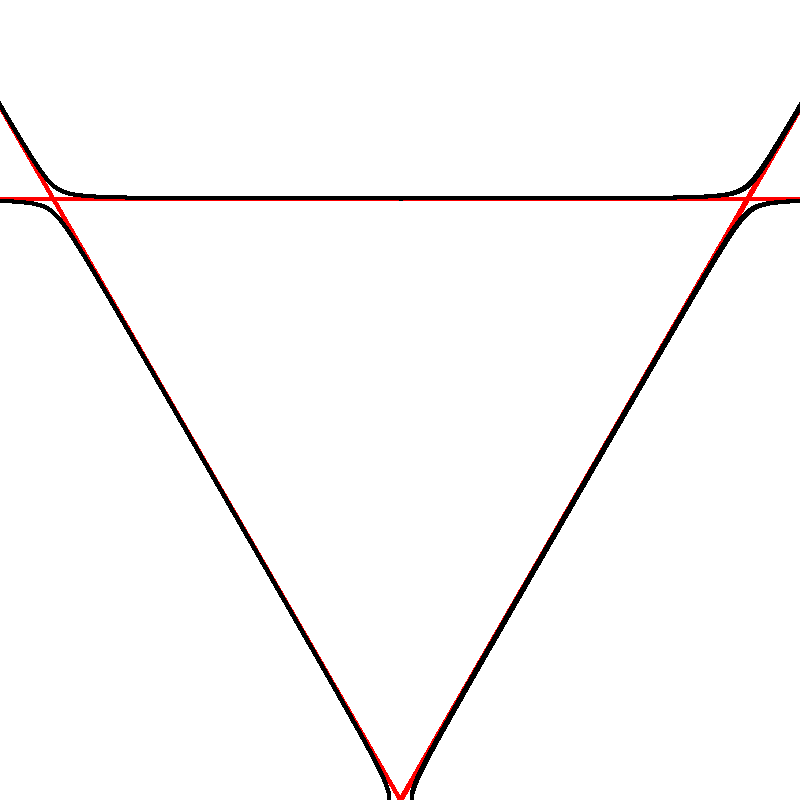}
\qquad\includegraphics[width=1.3in]{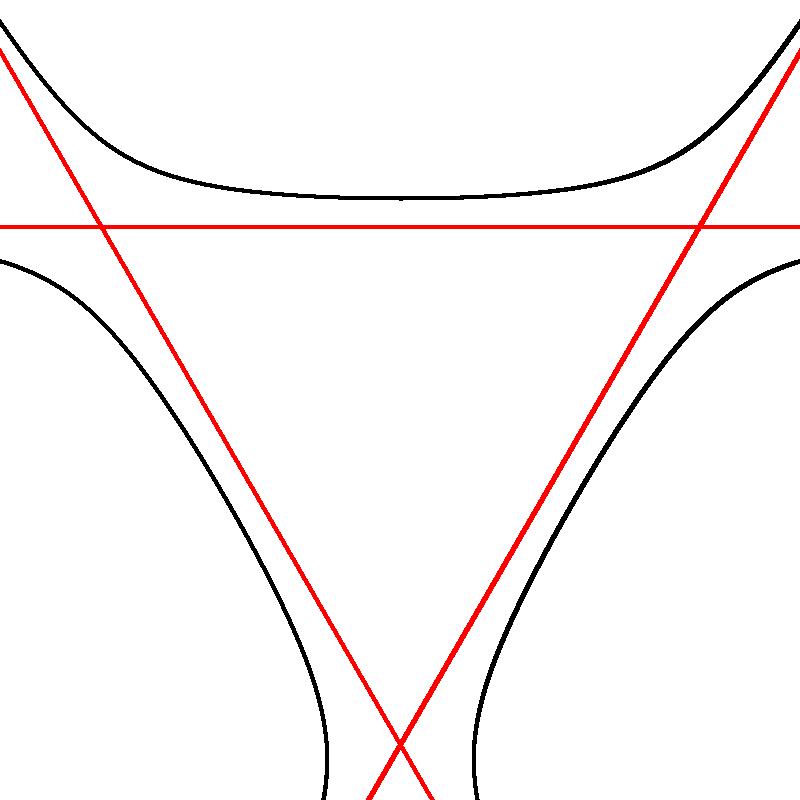}
\qquad\includegraphics[width=1.3in]{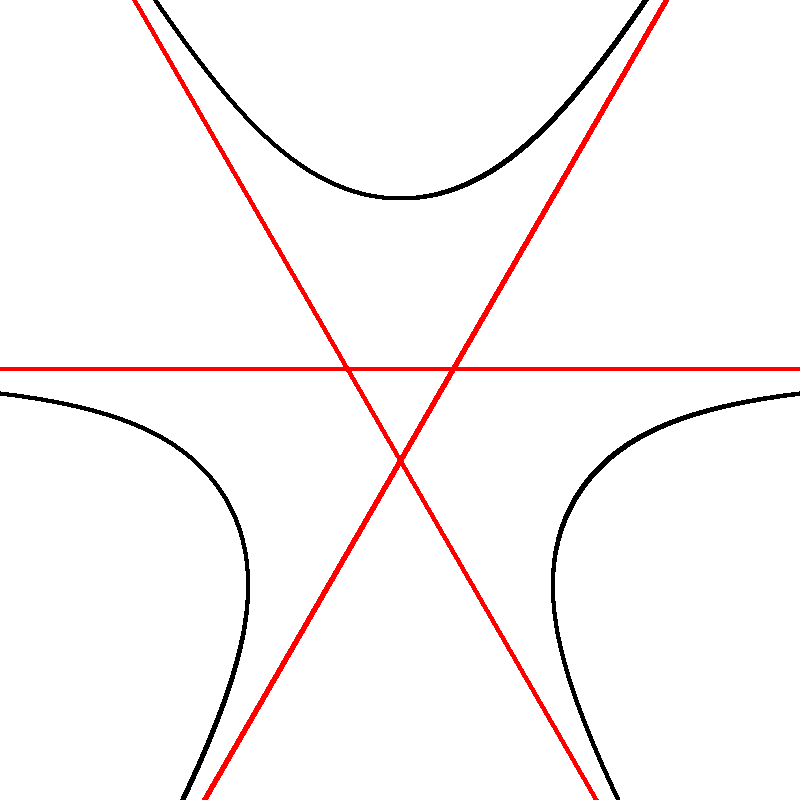}}\vspace{-.3cm}
$$k\approx -\infty\hspace{3.3in}$$
\bigskip

 \centerline{\includegraphics[width=1.3in]{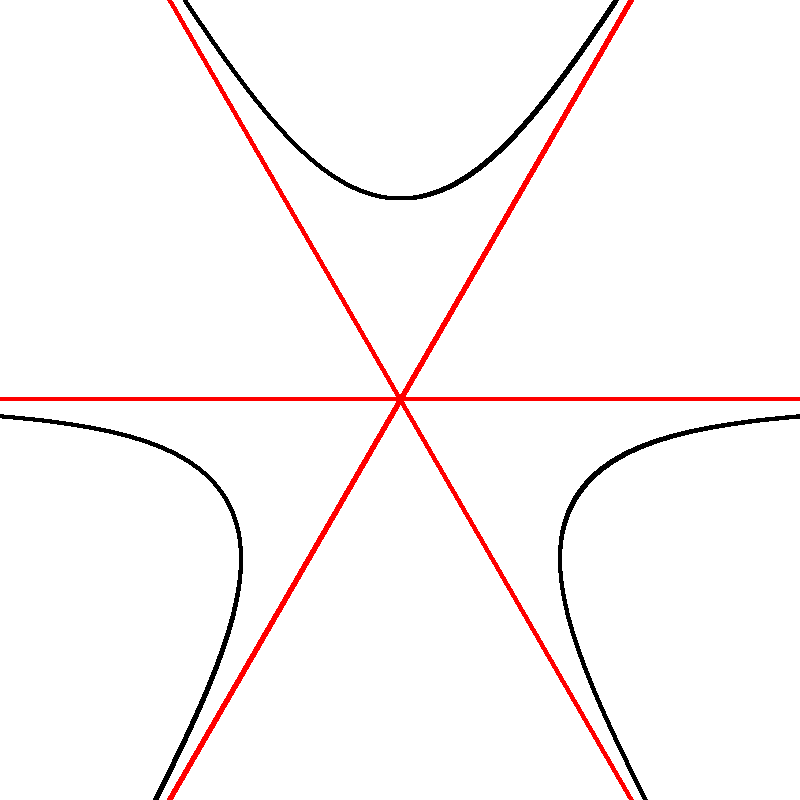}
\qquad\includegraphics[width=1.3in]{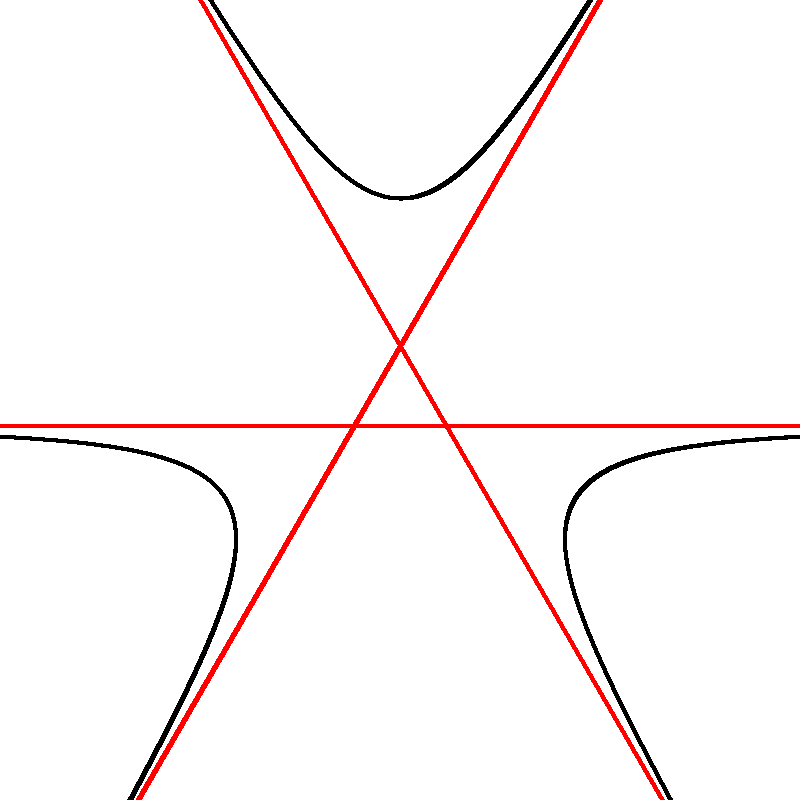}
\qquad\includegraphics[width=1.3in]{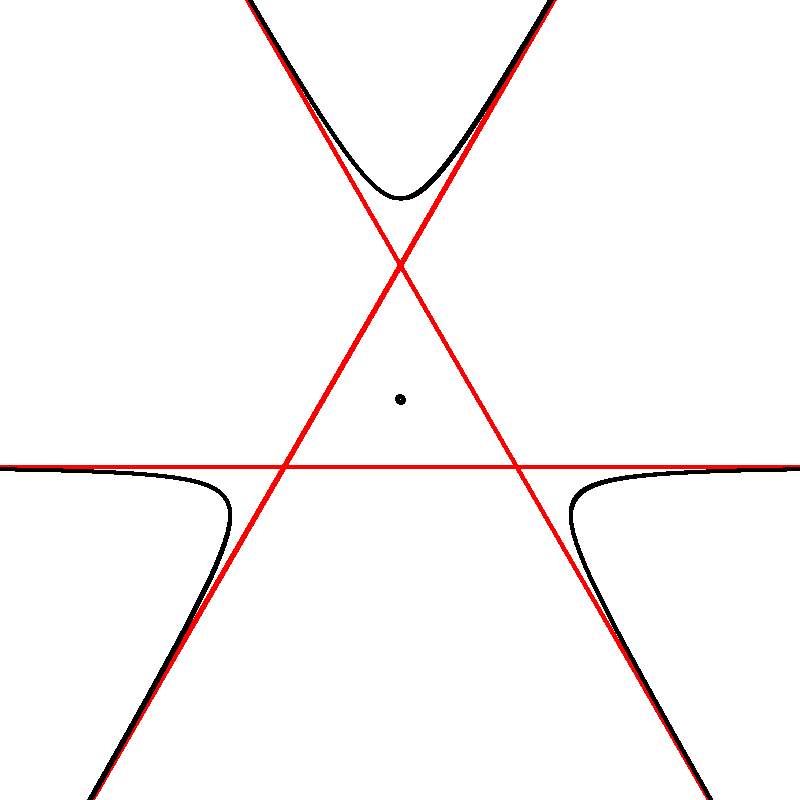}}\vspace{-.3cm}
$$k = -2  \hspace{3in} k = 1$$
\bigskip

 \centerline{\includegraphics[width=1.3in]{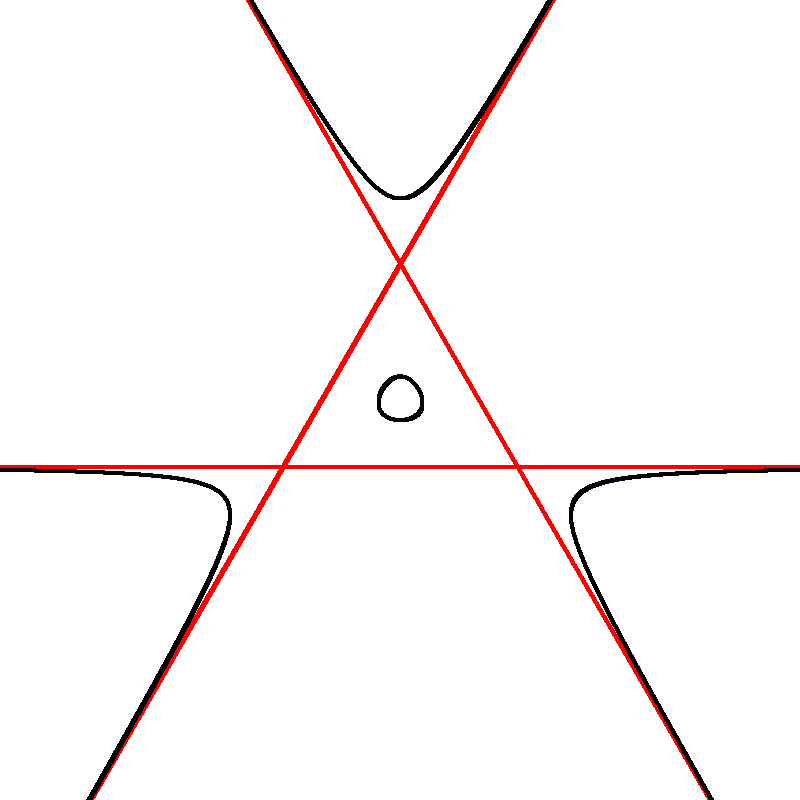}
\qquad\includegraphics[width=1.3in]{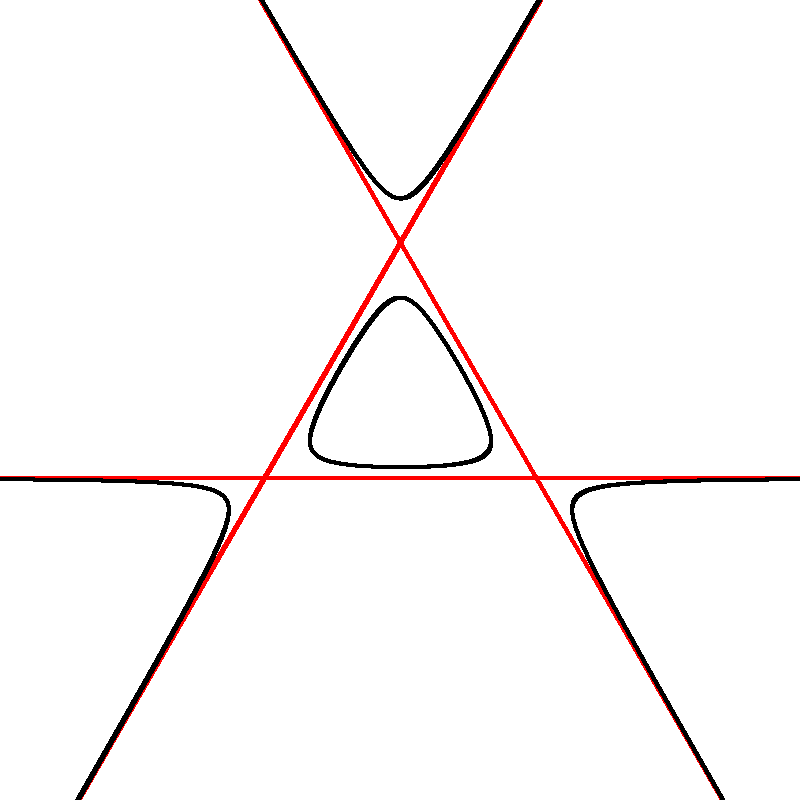}
\qquad\includegraphics[width=1.3in]{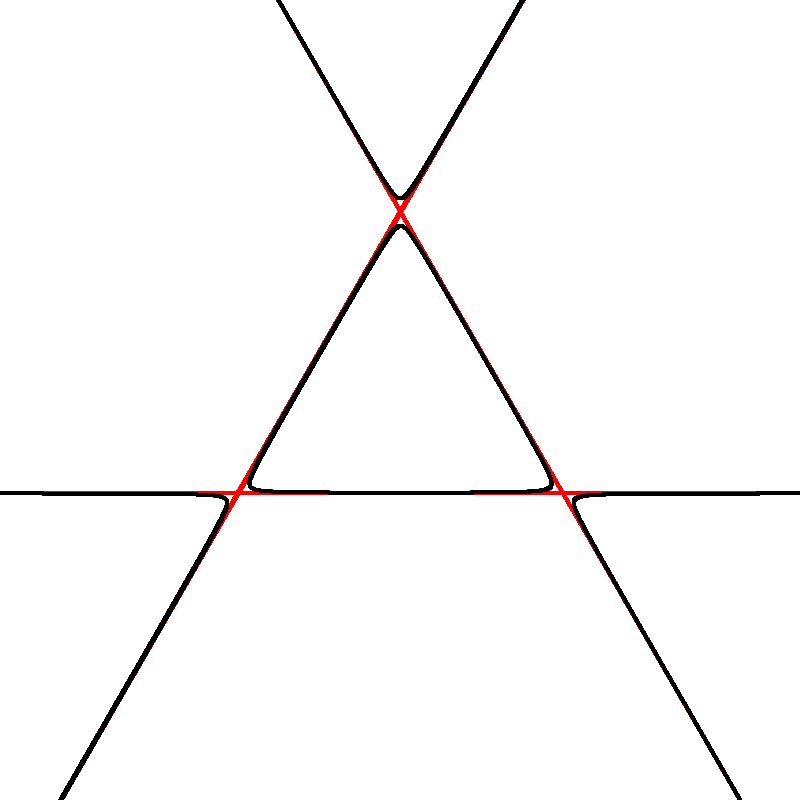}}\vspace{-.3cm}
$$\hspace{3.3in}k\approx +\infty$$
\mycaption{\label{F-canpics}\sf Nine pictures of real cubic curves
in canonical form, with Hesse invariant $k$ increasing from near $-\infty$
in the first picture, to near $+\infty$ in the last. Note that the
curve tends to a union of three straight lines as $k$ tends to $\pm\infty$.
The case $k=1$ is also singular, with an isolated point at the origin.
The case $k=-2$ (with $J=0$) 
is noteworthy, since this is the only case where
the three asymptotic lines meet at a common point.}
\end{figure}

\begin{rem}[\bf Visualizing Automorphisms]\label{R-vis}
If we use the standard normal form, or indeed almost any projectively
equivalent form, then the six automorphisms are very hard to visualize.
The picture becomes much clearer if
we choose a spherical metric for the projective plane which is
invariant under these automorphisms, as in Figure \ref{F-sym}. However,
it can still be confusing. For example, each of the three
involutions can be described either as a $180^\circ$ rotation about one
of the flex points (which lifts to an orientation preserving
rotation of the covering 2-sphere), or as a reflection about the line of
symmetry  (= great circle) which passes through the north-south pole, 
and crosses the equator halfway between the other two flex points. With the
second description, it evidently lifts to an orientation reversing reflection
of the 2-sphere.
\end{rem}
\medskip

\begin{rem}[\bf Canonical Position]\label{R-can}
Every real cubic curve can also be represented by a canonical picture
in the affine plane which makes
its six symmetries evident. Simply put
 the three flex points line at infinity, 
and put the center of symmetry at the origin. The tangent lines at the three
flex points will then appear as asymptotic lines. If we choose a Euclidean
metric so that the automorphisms are Euclidean isometries, then the picture
will be unique up to rotation and scale. Finally, we can choose a rotation
so that the reflection $(x,y)\leftrightarrow(-x,y)$ about the $y$ axis
is one of the automorphisms, and choose the scale so that $(0,1)$ is
the unique point
on the $y$-axis which belongs to the essential component of $\cC_\R$.
Then we will have a uniquely determined picture for each $\cC(k)_\R$. 
Some typical examples 
are shown in Figure \ref{F-canpics}.

As an extra bonus, this picture tends to a well defined limit as we approach
any one of the singular cases, at $k=1$ or $k=\pm\infty$. The limit as $k\to
1$ is a smooth curve plus an isolated point at the origin, while the
limit as $k\to\pm\infty$ is a union of three lines. 
\end{rem}

\begin{rem}\label{R-k>1} 
In the case $k>1$ when $\cC(k)_\R$ has two components, there is a direct 
geometric relationship between this canonical picture and the shape invariant
 of Proposition~\ref{P-tri}. Choose an axis of symmetry, for example the 
$y$-axis, in any of the  pictures in Figure~\ref{F-canpics}. Then the curve 
intersects this axis in three distinct points. As a fourth distinct point, 
choose the intersection point of this axis of symmetry with the horizontal 
asymptotic line. Labeling the coordinates of these points along the line
in order as $y_1,\,y_2,\,y_3,\,y_4$, we can form a variant of the cross-ratio:
$$ \chi~=~\frac{(y_1-y_4)(y_2-y_3)}{(y_1-y_2)(y_3-y_4)}~>~0~.$$
Now choose a projective equivalence between $\cC(k)_\R$ and a corresponding
 curve in standard normal form, with the axis of symmetry corresponding to the
 $x$-axis in standard coordinates.
Then the points $y_j$ will correspond to the points $r_1,\,r_2,\,r_3,\,\infty$,
where the $r_j$ are the roots of $x^3+ax+b$. Hence $\chi$ is equal to the 
cross-ratio $$ ~\chi~=~\frac{r_2-r_3}{r_1-r_2}~.$$ Now if we change the sign
 of the coefficient $b$, then we must rotate the complex $x$-plane by 
$180^\circ$. This will interchange $r_1$ and $r_3$, and hence replace $\chi$ by
 $1/\chi$. Inspecting Figure~\ref{F-canpics}, we see that $\chi$ tends to zero
 as $k\to 1$, and that $\chi$ tends to infinity as $k\to+\infty$.

\end{rem}

\end{document}